\documentclass[twoside]{article}
\usepackage{amsmath,amsfonts,amssymb,amsthm}
\usepackage{hyperref}
\usepackage[latin2]{inputenc}
\usepackage[a4paper]{geometry}
\usepackage[color,matrix,arrow,all]{xy}
\usepackage{stackrel}
\usepackage{mathdots}
\usepackage{kbordermatrix}
\usepackage{relsize}
\newcommand{\ds}{\displaystyle\sum}
\newcommand{\dip}{\displaystyle\prod}

\newcommand{\complex}{\mathbb{C}}
\newcommand{\nat}{\mathbb{N}}
\newcommand{\integ}{\mathbb{Z}}
\newcommand{\orb}{\mathcal{O}}
\newcommand{\op}{\operatorname}
\newcommand{\ul}{\underline}
\newcommand{\Hom}{\operatorname{Hom}}

\newcommand{\End}{\operatorname{End}}
\newcommand{\Ext}{\operatorname{Ext}}
\newcommand{\GL}{\operatorname{GL}}
\newcommand{\Rep}{\operatorname{Rep}}
\newcommand{\A}{\mathbb{A}}
\newcommand{\D}{\mathbb{D}}
\newcommand{\E}{\mathbb{E}}
\DeclareMathOperator{\Dim}{\mathbf{dim}}
\newtheorem{theorem}{Theorem}[section]

\newtheorem{lemma}[theorem]{Lemma}
\newtheorem{proposition}[theorem]{Proposition}
\newtheorem{corollary}[theorem]{Corollary}

\theoremstyle{definition}
\newtheorem{remark}[theorem]{Remark}
\newtheorem{definition}[theorem]{Definition}
\newtheorem{notation}[theorem]{Notation}

\newtheorem{example}[theorem]{Example}

\def\presuper#1#2%
  {\mathop{}%
   \mathopen{\vphantom{#2}}^{#1}%
   \kern-\scriptspace%
   #2}

\title{Decompositions of Bernstein-Sato polynomials and slices}
\author{Andr\'as Cristian L\H{o}rincz}
\date{}

\begin{document}

\maketitle

\begin{abstract}
Let $G$ be a linearly reductive group acting on a vector space $V$, and $f$ a (semi-)invariant polynomial on $V$. In this paper we study systematically decompositions of the Bernstein-Sato polynomial of $f$ in parallel with some representation-theoretic properties of the action of $G$ on $V$. We provide a technique based on a multiplicity one property, that we use to compute the Bernstein-Sato polynomials of several classical invariants in an elementary fashion. Furthermore, we derive a \lq\lq slice method"  which shows that the decomposition of $V$ as a representation of $G$ can induce a decomposition of the Bernstein-Sato polynomial of $f$ into a product of two Bernstein-Sato polynomials -- that of an ideal and that of a semi-invariant of smaller degree. Using the slice method, we compute Bernstein-Sato polynomials for a large class of semi-invariants of quivers. %Furthermore, we give a simple algorithm for describing the generic representation of a type $\D$ Dynkin quiver explicitly. We provide several examples to show how all computations can be carried out by hand.
\end{abstract}
%\tableofcontents

%\newpage
%\setcounter{section}{-1}

\section*{Introduction}
\label{sec:intro}

The classification of irreducible prehomogeneous vector spaces was achieved in  \cite{saki}. The computation of $b$-functions (i.e. Bernstein-Sato polynomials) of their semi-invariants has been completed using sophisticated methods such as microlocal calculus (for example, see \cite{kimu,skko}). Extensive calculations have been done also in the case of \emph{reducible} prehomogeneous vector spaces (for example, \cite{en,sasu, sugi, ukai}). 

In the article \cite{sasu}, a criterion has been given for the decomposition of the $b$-functions on a prehomogeneous space $V$ in terms of decomposing $V$ into smaller representations. Using this, the $b$-functions for quivers of type $\A$ are computed in \cite{sugi}. In this paper, we provide a more general computational technique based on a multiplicity one property that gives similar decompositions of $b$-functions. This technique gives a more elementary approach for the computation of the (local and global) $b$-functions of some classical semi-invariants, such as the determinant, symmetric determinant, Pfaffian and others. Furthermore, we derive a slice method leading to a reduction process that decomposes the $b$-function of a semi-invariant of $V$ into the product of the Bernstein-Sato polynomial of an ideal and the $b$-function of a semi-invariant on a slice of $V$. Applying this process, we can compute the $b$-functions for some semi-invariants of quivers, including those of Dynkin type $\A,\D$ and other tree quivers.

In \cite{en}, the author gives a method by \lq\lq reflections'' that allows the computation of $b$-functions for semi-invariants of any Dynkin quiver. For quivers, the slice method has the advantage of yielding faster results in most cases (when applicable). Also, the slice technique does not require extensive knowledge of representations of quivers. For best results in the case of quivers, the two methods can (and should) be combined. We note that in \cite{robin} the $b$-function for a semi-invariant of a special quiver (with a loop) is investigated using different tools. 

In his thesis \cite[Chapter 4]{phd}, the author considers a slice method similar to the one in this paper,  but which is rather cumbersome to use. The methods in this paper are major improvements of the slice method considered there.

We consider the following examples. Take $X=(x_{ij})$ an $n\times n$ generic matrix of variables, and $\partial X$ the matrix formed by the partial derivatives $\dfrac{\partial}{\partial x_{ij}}$. Its determinant is a differential operator. The classical Capelli identity implies (see \cite{howumed}):
$$\det \partial X \cdot \det X^{s+1} = (s+1)(s+2)\cdots(s+n)\det(X)^s.$$
Hence the Bernstein-Sato polynomial of the determinant is $b(s)=(s+1)(s+2)\cdots (s+n)$. In Section \ref{sec:slice} we explain how one can use the technique based on the multiplicity one property to derive this result in an elementary way.

A simple, yet non-trivial example of interest is the following semi-invariant, coming from the quiver $\D_4$:
$$\det\begin{pmatrix}
X & Y & 0 \\
0 & Y & Z
\end{pmatrix}.$$
Here $X,Y,Z$ are generic matrices of variables, with $X\in M_{\beta_4,\beta_1},Y\in M_{\beta_4,\beta_2},Z\in M_{\beta_4,\beta_3}$ and $\beta_1+\beta_2+\beta_3=2\beta_4$. We compute its $b$-function (together with many other quivers) in Section \ref{subsec:exquiv} based on the slice method developed in Section \ref{subsec:method}.

The paper is organized as follows. In Section \ref{sec:b-func}, we focus on generalities about Bernstein-Sato polynomials, mostly in the equivariant setting. 

In Section \ref{sec:slice}, we start by describing a method based on a multiplicity one property. We use Theorem \ref{thm:mult1} in order to compute the $b$-functions of several classical semi-invariants in Section \ref{subsec:exam}. Then we derive a slice method in Section \ref{subsec:method}, where the main result is Theorem \ref{thm:mainapp}. We also give the analogous result for $b$-functions of several variables (Theorem \ref{thm:multi}).

In Section \ref{sec:quiv}, after introducing some background material on quivers, we apply the slice method (Theorem \ref{thm:mainapp}) to arrows of quivers (Theorem \ref{thm:bquiv}). This gives a practical reduction method for computing $b$-functions of many (determinantal) quiver semi-invariants. This includes those of quivers of type $\A,\D$ and other tree quivers (see Theorems \ref{thm:tree}), \ref{thm:dee}. We work out several examples in Section \ref{subsec:exquiv} of $b$-functions of one variable and $b$-functions of several variables. %Using Proposition \ref{thm:neg}, we give an example of a semi-invariant of the Dynkin quiver $\E_6$ that is not sliceable.

Besides yielding the roots of $b$-functions, the slice method provides other useful information as well. For example, it gives an algorithm for determining the locally semi-simple representation corresponding to a semi-invariant (see Proposition \ref{prop:locsemi}). Based on slices, we also give an easy algorithm for the explicit description of generic representations for Dynkin quivers of type $\D$, as described in Appendix \ref{app:decomp}. 

\begin{notation}\label{not:notation}
As usual, $\nat$ will denote the set of all non-negative integers and $\complex$ the set of complex numbers. For $a,b,d\in \nat, a\leq b$, we use the following notation in $\complex[s]$: 
$$[s]^d_{a,b}:=\prod_{i=a+1}^b \prod_{j=0}^{d-1} (ds+i+j).$$
In the case $d=1$, we sometimes write $[s]_{a,b}:=[s]_{a,b}^1$. Also, if $a=0$, we sometimes write $[s]^d_{b}:=[s]^d_{0,b}$. Hence $[s]^{d}_{a,b}[s]^d_{a}=[s]^d_b$.

Now fix an $l$-tuple $\underline{m}=(m_1,\dots,m_l)\in \nat^l$. Then for any $l$-tuple $(d_1,\dots,d_l)$, we use the following notation in $\complex[s_1,\dots,s_l]$:
$$[s]^{d_1,\dots,d_l}_{a,b}=\prod_{i=a+1}^b \prod_{j=0}^{d-1} (d_1s_1+\dots+d_ls_l+i+j),$$
where $d=m_1d_1+\dots + m_ld_l$.
\end{notation}

\section{Bernstein-Sato polynomials}
\label{sec:b-func}

\subsection{Definition} \label{subsec:bern}

First we define and briefly recall some basic properties about Bernstein-Sato polynomials. We will interchangeably call them also $b$-functions, especially in the contexts of Theorem \ref{thm:bfun} and Lemma \ref{lem:several} from Section \ref{subsec:bfun}. For details on Bernstein-Sato polynomials, we refer the reader to \cite{gyoja,kashi}. 

Throughout this paper we work over the complex field $\complex$. Let $V$ be an $n$-dimensional vector space. Denote by $D$ the algebra of differential operators on $V$ (i.e. the Weyl algebra in $n$ variables), and by $D_v$ the algebra of differential operators regular at $v\in V$ (i.e. the localization of $D$ at $v$).

Let $f\in \complex[V]$ be a non-zero polynomial, and let $R$ be one of the rings $D$ or $D_v$. Then there exits (see \cite{kashi} a differential operator $P(s)\in R[s]:=R\otimes \complex[s]$ and a non-zero polynomial $b(s)\in \complex[s]$ such that 
$$P(s)\cdot f^{s+1}(x) = b(s)\cdot f^s(x).$$
The functions $b(s)$ satisfying such a relation form an ideal of $\complex[s]$, whose monic generator we denote by $b_f(s)$ or $b_{f,v}(s)$, if $R=D$ or $D_v$, respectively. We call $b_f$ the (global) \textit{Bernstein-Sato polynomial} of $f$, and $b_{f,v}$ the local Bernstein-Sato polynomial of $f$ at $v$. 

By \cite{kashi}, all roots of $b_f(s)$ are negative rational numbers. Moreover, if $f$ is a homogeneous polynomial, then $b_{f,0}(s)= b_f(s)$ (see \cite[Lemma 2.5.3]{gyoja}).

Throughout we work mostly in equivariant settings as seen in the next section.

\subsection{$b$-functions of semi-invariants}\label{subsec:bfun}

Let $G$ be a (connected) reductive algebraic group, acting rationally on $V$. That is, we have a morphism of algebraic groups $\rho: G \to \GL(V)$. Then we have an action of $G$ on $\complex[V]$ by $(g\cdot f)(v)=f(g^{-1}\cdot v)$ for all $v\in V$, where $g\in G$, $f\in \complex[V]$. We call a polynomial $f\in \complex[V]$ a semi-invariant, if there is a character $\sigma\in \Hom(G,\complex^{\times})$ such that $g\cdot f = \sigma(g) f$, that is, $f(gv)=\sigma(g)^{-1}f(v)$. In this case we say the weight of $f$ is $\sigma$. In the literature such $f$ is sometimes also called a relative invariant polynomial. We form the ring of semi-invariants 
$$\op{SI}(G,V)=\bigoplus_{\sigma} \op{SI}(G,V)_{\sigma}=\complex[V]^{[G,G]},$$
where the sum runs over all characters $\sigma$ and the weight spaces are
$$\op{SI}(G,V)_{\sigma}=\{f\in \complex[V] | f \text{ is a semi-invariant of weight } \sigma\}.$$
The multiplicity of $\sigma$ is $\dim\op{SI}(G,V)_{\sigma}$. Following \cite{en}, we make the following definition (which makes sense even when $G$ is not reductive): 
\begin{definition}\label{def:multfree}
We say that $\sigma$ is \emph{multiplicity-free}, if the multiplicity of $\sigma^k$ is $1$, for any $k\in\nat$.
\end{definition}

By a standard argument, one can give the following geometric characterization of the above property when $G$ is a connected reductive group: a semi-invariant $f\in \complex[V]$ has multiplicity-free weight $\sigma$ if and only if there is a unique closed orbit $\orb$ in the open affine neighborhood $f\neq 0$. In the spirit of \cite{shme}, for an element $x\in \orb$ of this orbit we say that $x$ is the \textit{locally semi-simple} point of $f$.

Given a semi-invariant $f$ of weight $\sigma$, for the results of this paper regarding $b$-functions to hold (see Theorem \ref{thm:bfun})  it is enough to require the multiplicity of $\sigma^k $ to be $1$ for just $k=\deg f-1$.

Let $x=(x_1,\dots, x_n)$ be the coordinate system with respect to a basis of $V$. %Similarly, let $\{e^*_1,\dots , e^*_n\}$ be the dual basis of $\{e_1,\dots , e_n\}$ and $x^*=(x^*_1,\dots,x_n^*)$ be the coordinate system of $V^*$ with respect to the dual basis. 
We denote the dual variables (partial derivatives) by
\[\partial x=(\partial_1, \dots, \partial_n).\]
Let $V^*$ be the dual space of $V$, with is naturally a $\GL(V)$-module. For any $d\geq 0$, let $\complex[V]_d$ (resp. $\complex[V^*]_d$) be the subspaces of homogeneous polynomials of degree $d$ in $\complex[V]$ (resp. $\complex[V^*]$). We have the $\GL(V)$-equivariant pairing between $\complex[V]_d$ and $\complex[V^*]_d$ by
\begin{equation}\label{eq:pair}
\langle P,P^*\rangle = P^*(\partial x) \cdot P(x).
\end{equation}
This gives a $\GL(V)$-equivariant isomorphism $\complex[V^*]_d \cong (\complex[V]_d)^*$.

Let $f\in \complex[V]$ be a semi-invariant of weight $\sigma$, and assume $\sigma$ is multiplicity-free. Then $f$ must be homogeneous (see \cite[Lemma 1.3]{gyoja}). Since $G$ is reductive, by the above pairing there is a dual semi-invariant $f^*\in \complex[V^*]$ of weight $\sigma^{-1}$ of the same degree, canonical up to constant. In fact, we can choose a basis of $V$ such that the subset $\rho(G)\subset \GL(V)$ is stable under conjugate transpose, in which case $f^*$ can be obtained from $f$ by taking the complex conjugates of the coefficients -- see \cite{saki}.
The next result follows by \cite[Lemma 1.6,1.7]{gyoja} and \cite[Corollary 2.5.10]{gyoja}.

\begin{theorem}\label{thm:bfun}
Let $f\in \complex[V]$ be a semi-invariant with multiplicity-free weight, and let $f^*\in\complex[V^*]$ be the dual semi-invariant. We have
\begin{equation}\label{eq:good}
f^*(\partial x)\cdot f(x)^{s+1}=b(s)f(x)^s.
\end{equation}
where $b(s)$ is a polynomial equal to the Bernstein-Sato polynomial $b_f(s)$ up to a non-zero constant factor and $\deg b_f(s) = \deg f$.
\end{theorem}

%A general goal is to compute the $b$-function $b(s)$ from the differential equation above. 

We call $(G,V)$ a prehomogeneous vector space, if $V$ has a dense open orbit $\orb$, i.e. $\overline{\orb}=V$. By Rosenlicht's Theorem (see \cite{preh}), $(G,V)$ is prehomogeneous iff all weight multiplicities of the ring of semi-invariants are at most $1$. Moreover, the following holds (see \cite{saki}):
\begin{theorem}\label{thm:satokimura}
Assume $(G,V)$ is a prehomogeneous vector space, and let $Z(f_1),Z(f_2),\dots, Z(f_k)$ be the irreducible components of $V\backslash \orb$ of codimension $1$, for some $f_1,f_2,\dots,f_k \in \complex[V]$. Then $f_1,f_2\dots,f_k$ are algebraically independent semi-invariants and $\op{SI}(G,V)=\complex[f_1,f_2,\dots,f_k]$.
\end{theorem}
The semi-invariants $f_1,f_2\dots,f_k$ as above are called \textit{fundamental} semi-invariants. We mention that many of our examples in this paper are prehomogeneous vector spaces, but we also work with spaces that are not necessarily prehomogeneous but have semi-invariants of multiplicity-free weights (for example, Theorem \ref{thm:tree}).

We have the following notion of $b$-function of several variables (see \cite{sata}).Let $f_1,\dots f_l \in \complex[V]$ be semi-invariants of weights $\sigma_1,\dots, \sigma_l$, respectively. Assume that the product $\sigma_1\cdots \sigma_l$ is a multiplicity-free weight in $\complex[V]$. In this case we can take respective dual semi-invariants $f^*_1,\dots, f^*_l\in \complex[V^*]$. Put $\underline{f}=(f_1,\dots, f_l)$ and $\underline{f}^*=(f_1^*,\dots , f_l^*)$. For a multi-variable $\underline{s}=(s_1,\dots ,s_l)$, we define $\underline{f}^{\underline{s}}=\dip_{i=1}^l f_i^{s_i}$, and $\underline{f}^{*\underline{s}}=\dip_{i=1}^l f_i^{*s_i}$.

\begin{lemma}\label{lem:several}
Using the notation above, if $\sigma_1\cdots \sigma_l$ is multiplicity-free, then for any $l$-tuple $\underline{m}=(m_1,\dots,m_l)\in \nat^l$ there is a polynomial $b_{\ul{f},\underline{m}}(\underline{s})$ of $l$ variables such that 
\begin{equation}
\underline{f}^{*\underline{m}}(\partial x)\cdot\underline{f}^{\underline{s}+\underline{m}}(x)=b_{\ul{f},\underline{m}}(\underline{s})\underline{f}^{\underline{m}}(x).
\end{equation}
\end{lemma}

If  $\sigma_1\cdots \sigma_l$ is multiplicity-free, then all the individual weights $\sigma_i$ are multiplicity-free, and one can easily recover the $b$-function $b_{f_i}(s)$ of one variable from $b_{\ul{f},\underline{m}}(\ul{s})$. Again, if $(G,V)$ is prehomogeneous then any $\sigma_1\cdots \sigma_l$ is automatically multiplicity-free.

\subsection{Bernstein-Sato polynomials of ideals}\label{subsec:bideal}

Now we consider tuples of polynomials $\ul{f}=(f_1,\dots,f_r)$ with $f_i \in \complex[V]$, from a different viewpoint. Following \cite[Definition 3.3]{bmax}, we introduce (note that in the case of $r=1$ we recover Definition \ref{def:multfree}):

\begin{definition}\label{def:tuple}
A tuple $\ul{f}= (f_1,\dots,f_r)$ in $\complex[V]$ is said to be a \textit{multiplicity-free tuple} if
\begin{itemize}
\item[(a)] For every $k\in \nat$, the polynomials
\[\ul{f}^{\ul{k}} = f_1^{k_1}\cdots f_r^{k_r},\mbox{ for } \ul{k}=(k_1,\dots,k_r)\in\nat^r \mbox{ satisfying } k_1+\dots+k_r=k,\]
span an irreducible $G$-subrepresentation $M_{k}\subset \complex[V]$.
\item[(b)] For every $k\in\nat$, the multiplicity of the $G$-representation $M_{k}$ inside $\complex[V]$ is equal to one.
\end{itemize}
\end{definition}

We note that given any multiplicity-free tuple $\ul{f}=(f_1,\dots,f_r)$, any \lq\lq power'' of the tuple $\ul{f}$ is also multiplicity-free. Here the $d$th power of the tuple $\ul{f}$ is a new tuple formed by all elements of the form
\[\ul{f}^{\ul{d}} = f_1^{d_1}\cdots f_r^{d_r},\mbox{ for } \ul{d}=(d_1,\dots,d_r)\in\nat^r \mbox{ satisfying } d_1+\dots+d_r=d.\]

Now fix a multiplicity-free tuple $\ul{f} = (f_1,\dots,f_r)$, which WLOG we assume that is a basis of $M_1$. Since $G$ is reducitive and the multiplicity of $M_1$ is in $\complex[V]$ is one, there is a dual representation $M_1^*$ in $\complex[V^*]$ of multiplicity one. We take a basis $f_1^*,\dots,f_r^*$ that is $G$-dual (up to constant) to $f_1,\dots,f_r$ with respect to the pairing (\ref{eq:pair}). Then the element 
\[D_{\ul f}=\ds_{i=1}^r f^*_i(\partial x) f_i(x)\]
is a $G$-invariant differential operator. Denote by $I$ the ideal generated by $f_1,\dots,f_r$ in $\complex[V]$, and let $b_I(s)=b_{\ul{f}}(s)$ be the Bernstein-Sato polynomial of $I$ -- for the definition of Bernstein-Sato polynomials of ideals (or tuples), we refer the reader to \cite{bideal}.  By \cite[Proposition 3.4]{bmax}, we have the following result.

\begin{proposition}\label{prop:pf}
Consider a multiplicity-free tuple $\ul{f}=(f_1,\dots,f_r)$. If we let $s=s_1+\dots+s_r$ then there exists a polynomial $P_f(s)\in\complex[s]$ such that
\[D_{\ul{f}}\cdot \ul{f}^{\ul s} = P_{\ul f}(s)\cdot \ul{f}^{\ul s},\]
and the Bernstein-Sato polynomial $b_{\ul{f}}(s)$ divides $P_{\ul{f}}(s)$.
\end{proposition}

As in the case $r=1$ (by Theorem \ref{thm:bfun}), we conjecture that for multiplicity-free tuples $\ul{f}$ we always have equality $b_{\ul{f}}(s)=P_{\ul{f}}(s)$. In \cite{bmax} this has been shown to be the case when $I$ is the ideal generated by maximal minors or the ideal generated by sub-maximal Pfaffians. We can also consider powers of ideals $I^d$, for positive integers $d$, as follows. 

Let $M_{m,n}$ be the space of $m\times n$ matrices with $m\leq n$. Let $X$ be the $m\times n$ generic matrix of indeterminates and denote by $I$ the ideal of $\complex[V]$ generated by all the $n\times n$ minors of $X$.

\begin{theorem}\label{thm:powermax}
Let $I^d$ the power of the ideal $I$ generated by maximal minors for some $d\in \nat$. Then the Bernstein-Sato polynomial of $I^d$ is
\[b_{I^d}(s) = \prod_{i=n-m+1}^n \prod_{j=0}^{d-1} \left(s+\dfrac{i+j}{d}\right).\]
\end{theorem}

\begin{proof}
Consider the tuple formed by all maximal minors, which is a multiplicity-free tuple (see \cite{bmax}) by the FFT (see \cite[XI. Section 1.2]{proc}). Denote by $\ul{f}$ the $d$th power of this tuple as explained above, so $\ul{f}$ is also multiplicity-free, hence Proposition \ref{prop:pf} applies. One can obtain $P_{\ul f}(s)$ in several ways. For example, we can apply either the method in the proof of \cite[Theorem 3.5]{bmax} using the Fourier transform, or observe that by Schur's Lemma the polynomial $P_{\ul f}(s)$ is the same as the one computed in \cite[Theorem 3.3]{sasu} -- see also proof of Lemma \ref{lem:firstfac}. Hence, up to a constant we have (see Notation \ref{not:notation})
\[P_{\ul f}(s)=[s]^d_{n-m,n}.\]
To see that $b_{\ul{f}}(s)=P_{\ul{f}}(s)$, we note that the proof for the case $d=1$ from \cite[Section 4]{bmax} carries over, \textit{mutatis mutandis}, for an arbitary $d\in\nat$.
\end{proof}

\section{Slices and the multiplicity one property}
\label{sec:slice}

In this section, we develop several techniques for calculating $b$-functions. These are similar to the methods used in \cite{sasu,ukai,wach}. The slice method developed in Section \ref{subsec:method} will be used further in Section \ref{sec:quiv}.

\subsection{Slices}\label{subsec:slice} 

Let $H$ be a connected affine algebraic group and $V$ a rational $H$-module. Let $f\in \complex[V]$ be a non-zero $H$-semi-invariant of weight $\sigma$. Denote by $\mathfrak{h}$ the Lie algebra of $H$. Fix an element $v\in V$ and let $H_v$ be the stabilizer of $v$. The tangent space at $v$ to the orbit $\orb=G\cdot v$ of $v$ is $T_v(\orb)=\mathfrak{h}\cdot v$, on which $H_v$ acts naturally. By a theorem of Mostow \cite{mostow}, we can write $H_v= L_v \ltimes U_v$, where $U_v$ is the unipotent radical of $H_v$ and $L_v \cong H_v/U_v$ is reductive. Let $W$ be an $L_v$-complement to $T_v(\orb)$ in $V$, so that we have an $L_v$-decomposition $V=T_v(\orb)\oplus W$. We call $(L_v,W)$ the \textit{slice representation} at $v$.

Given a polynomial $f\in \complex[V]$, we construct a polynomial $f_v \in \complex[W]$ defined by $f_v(w):=f(v+w)$ for $w\in W$. This gives an algebra map from $\complex[V]$ to $\complex[W]$ given by $f \mapsto f_v$.
 
Now if $f\in \complex[V]$ is a $H$-semi-invariant of weight $\sigma$, then $f_v\in\complex[W]$ is a $L_v$-semi-invariant of weight $\sigma|_{L_v}$. Hence the map $f \mapsto f_v$ induces the maps
\begin{equation}\label{eq:slice}
\phi_v : \op{SI}(H,V) \to \op{SI}(L_v,W)\,\, , \,\, \phi_v^\sigma : \op{SI}(H,V)_{\sigma} \to \op{SI}(L_v,W)_{\sigma |_{L_v}}.
\end{equation}
As in \cite{ukai}, we consider the map 
$$\mu: H\times W \to V,$$
$$\mu(h,w)=h(v+w).$$
Computing the differential at the identity of $H$, we see that $\mu$ is a smooth map. In particular, the algebra map $\mu^*: \complex[V] \to \complex[H]\otimes\complex[W]$ is injective. The map separates variables for a semi-invariant $f$ of weight $\sigma$, for we have  
\begin{equation}
\mu^*(f)=\sigma^{-1}\otimes f_v.
\label{eq:basic}
\end{equation}
By the above discussion we obtain the following lemma (see also \cite[p. 57]{ukai}):

\begin{lemma}\label{lem:basic}
The map $\phi_v^\sigma$ is injective. Moreover, if $f$ is a semi-invariant of $(H,V)$ then $b_{f,v} = b_{f_v,0}$, that is, the local $b$-functions of $f$ at $v$ and of $f_v$ at $0$ coincide. In particular, if $f_v$ is homogeneous then $b_{f_v}|b_f$.
\end{lemma}

\begin{remark}\label{rem:works}
We note that in some situations one can choose algebraic groups (with corresponding complements $W$) different from $L_v$ and still make the above considerations work.
\end{remark}

\subsection{Expansions and the multiplicity one property}\label{subsec:mult1}

We recall and generalize some considerations from \cite{sasu}. Let $G$ be a (connected) reductive group with a Borel subgroup $B$ that contains a maximal torus $T$. The irreducible rational $G$-modules are parameterized by dominant $T$-weights. Let $V$ an algebraic $G$-module, and fix $f\in \op{SI}(G,V)_\sigma$ with $\sigma$ \textit{multiplicity-free} as in Definition \ref{def:multfree}. Then $f$ is homogeneous, say of degree $d>0$. Take any integer $k$ with $0< k < n$. We have a $G$-equivariant map 
\[ \complex[V]_k \otimes \complex[V]_{d-k} \to \complex[V]_d.\]
The polynomial $f$ lies in the image of this onto map. Decomposing $\complex[V]_k$ (resp. $\complex[V]_{d-k}$) into irreducible $G$-modules and using that the multiplicity of $\sigma$ in $\complex[V]$ is one, we see that there exits an irreducible $G$-submodule $M_\lambda$ of $\complex[V]_k$ (resp. $M_{\lambda^*\cdot \sigma}$ of $\complex[V]_{d-k}$) such that $f$ is in the image of the multiplication map
\[M_\lambda \otimes M_{\lambda^*\cdot\sigma} \to \complex[V]_d.\]
Here $M_\lambda$ (resp. $M_{\lambda^*\cdot\sigma}$) is an irreducible $G$-module of highest weight $\lambda$ (resp. $\lambda^* \cdot \sigma$) for some dominant weight $\lambda$, and $M_{\lambda^*\cdot \sigma}$ is $G$-isomorphic to the dual space of $M_\lambda$ tensored with the character $\sigma$. Take a basis $f_1^{(1)},\dots, f_p^{(1)}$ of $M_\lambda$, and take a $G$-dual basis $f_1^{(2)},\dots,f_p^{(2)}$ of $M_{\sigma\cdot\lambda^*}$. The above shows that we have an \textit{expansion} (up to non-zero constant)
\begin{equation}\label{eq:expand}
f(x) = \ds_{i=1}^p  f_i^{(1)}(x)f_i^{(2)}(x).
\end{equation}

In order to determine $M_\lambda \subset \complex[V]_k$ for some fixed $k$, we discuss the following typical examples. 

\begin{example}\label{ex:subvar}
Let $f_1^{(1)},\dots, f_p^{(1)}$ be a basis of an irreducible submodule $M_\lambda$ of $\complex[V]_k$. If $f$ lies in the ideal generated by $f_1^{(1)},\dots, f_p^{(1)}$ in  $\complex[V]$, then we have an expansion (\ref{eq:expand}) as above. Geometrically, if $f_1^{(1)},\dots, f_p^{(1)}$ generate the (reduced) defining ideal of a closed subset of the zero-set $Z(f)$, then we have an expansion (\ref{eq:expand}) as above.
\end{example}

\begin{example}\label{ex:main}
The case considered in \cite{sasu} is when $V$ is reducible, that is, there is a non-trivial $G$-decomposition $V=E\oplus F$. Then $\complex[V]=\complex[E]\otimes \complex[F]$, and we can choose $M_\lambda$ (resp. $M_{\lambda^*\cdot \sigma})$ to be a $G$-irreducible isotypic component $\complex[E]_\lambda$ (resp. $\complex[F]_{\lambda^*\cdot \sigma}$), for a unique dominant weight $\lambda$ (see \cite[Proposition 1.6]{sasu}). We remark that in \cite{sasu} the roles of $E$ and $F$ are interchanged.
\end{example}

Since $G$ is reductive, the constructions above can be obtained for $\complex[V^*]$ as well (see also \cite{sasu}). Namely, let $f^*\in \complex[V^*]_d$ be the dual semi-invariant of $f$, which then has multiplicity-free weight $\sigma^{-1}$. Under the assumptions above, there exists an irreducible $G$-submodule $N_{\lambda^*}$ of $\complex[V^*]_k$ that is $G$-isomorphic to the dual of $M_{\lambda}$, and an irreducible $G$-submodule $N_{\lambda\cdot \sigma^{-1}}$ of $\complex[V^*]_{d-k}$ that is $G$-isomorphic to the dual of $M_{\lambda^*\cdot \sigma}$, such that $f^*$ is in the image of the map
\[N_{\lambda^*} \otimes N_{\lambda \cdot \sigma^{-1}} \to \complex[V]_d.\]
 Then we have an expansion of the form
\[f(x^*) = \ds_{i=1}^p  f_i^{*(1)}(x^*)f_i^{*(2)}(x^*),\]
for $x^* \in V^*$. Here we can take  $f_1^{*(1)},\dots, f_p^{*(1)}$ (resp. $f_1^{*(2)},\dots,f_p^{*(2)}$) to be a basis of $N_{\lambda^*}$ (resp. $N_{\lambda\cdot \sigma^{-1}}$) that is $G$-dual to $f_1^{(1)},\dots, f_p^{(1)}$ (resp. $f_1^{(2)},\dots,f_p^{(2)}$) with respect to the pairing \ref{eq:pair}.

As in \cite{sasu}, we assume that the following \textit{multiplicity one property} is satisfied: $\complex[V]_{\lambda\cdot \sigma^{d-k-1}}=M_\lambda\cdot f^{d-k-1}$, or equivalently:
\begin{equation}\label{eq:mult1}
\mbox{The multiplicity of the irreducible $G$-module of highest weight } \lambda \cdot \sigma^{d-k-1} \mbox{ in } \complex[V] \mbox{ is }1.
\end{equation}

We obtain the following generalization of \cite[Theorem 1.12]{sasu} (the proof is analogous):

\begin{theorem}\label{thm:mult1}
Let $f$ be a semi-invariant with multiplicity-free weight, and take an expansion (\ref{eq:expand}) as above. Assume that the multiplicity one property (\ref{eq:mult1}) holds. Then the $b$-function of $f$ decomposes as $b_f(s)=b_1(s)\cdot b_2(s)$ with:
\begin{enumerate}
\item[(1)]  $\,\,\left[\ds_{i=1}^p f^{*(1)}_i(\partial x) f^{(1)}_i (x)\right] \cdot f^{s}(x) = b_1(s) f^s(x),$
\item[(2)]  $\,\,f^{*(2)}_i(\partial x) \cdot f^{s+1}(x) = b_2(s) f^{(1)}_i (x) f^s(x)\,\,$ (for any  $i=1,\dots ,p$).
\end{enumerate}
\end{theorem}

\begin{remark}\label{rem:loc}
We note that if $v$ is any element in $V$ with $f_i^{(1)}(v)\neq 0$ (for some $i$) then equation (2) above is a candidate for giving the local $b$-function of $f$ at $v$. In other words, $b_{f,v}(s) | b_2(s)$. In fact, we will see that in some situations equality holds, and that $b_2(s)$ can be itself a $b$-function of a semi-invariant of lower degree -- see Sections \ref{subsec:exam}, \ref{subsec:method}.
\end{remark}

Now we discuss the $k=1$ case for Theorem \ref{thm:mult1} in more detail:

\begin{corollary}\label{cor:euler}
Assume $(G,V)$ is an irreducible prehomogeneous vector space and $f\in \complex[V]$ a semi-invariant of weight $\sigma$. Let $n=\dim V$ and $d=\deg f>1$, and assume the multiplicity of the irreducible representation $V^*\otimes \sigma^{d-2}$ in $\complex[V]$ is one. Then $-n/d$ is a root of $b_f(s)/(s+1)$.
\end{corollary} 

\begin{proof}
The multiplicity one property (\ref{eq:mult1}) holds, where $k=1$ and $M_\lambda=V^* = \complex[V]_1$. By Theorem \ref{thm:mult1}, we have a decomposition $b_f(s) = b_1(s) \cdot b_2(s)$. Clearly, $-1$ is a root of $b_2(s)$, and $b_1(s)$ satisfies
\[\left(\ds_{i=1}^n \partial_{i} x_{i}\right) \cdot f^s = b_{1}(s) \cdot f^s.\]
The operator on the LHS equals $E+n$, where $E$ denotes the usual Euler operator.  Hence, we have $b_1(s)=ds+n$, proving our claim. 
\end{proof}

We note that for all irreducible prehomogeneous spaces considered in \cite{kimu}, $-n/d$ is indeed a root of the $b$-function, suggesting that the multiplicity-one property holds frequently among these (see examples in the next section).

\subsection{Examples of irreducible prehomogeneous spaces}\label{subsec:exam}

As explained in \cite[Section 3.1]{sasu}, the decomposition technique as in Example \ref{ex:main} can be used to obtain in an elementary way the $b$-functions of some classical (semi-)invariants such as the determinant and the Pfaffian. Previous proofs rely on sophisticated methods such as Capelli's identity (see \cite{howumed,proc}) or microlocal calculus (see \cite{kimu}). However, for the calculation of the $b$-function of the symmetric determinant, the technique as in Example \ref{ex:main} is not sufficient. As it turns out, considering more general expansions as (\ref{eq:expand})  is adequate for this purpose. Furthermore, in combination with methods from Section \ref{subsec:slice}, we obtain all the local $b$-functions of these classical invariants as well. For illustration, we now work out the case of the symmetric determinant and several others that do not arise from reducible representations as in Example \ref{ex:main}. These suggest that many $b$-functions of semi-invariants of prehomogeneous vector spaces can be computed with this method.  Further examples will be provided for semi-invariants of quivers (Section \ref{sec:quiv}). 

For the standard notation that we use for the representations below, cf. \cite{saki}. 

\begin{example} \label{ex:symdet} $(\GL(n), 2\Lambda_1)$, the symmetric determinant.

We can think of elements $M\in V=\op{Sym}^2 \complex^n$ as symmetric matrices $M=M^t$, on which the action of $G=\GL(n)$ is given by $g\cdot M = gMg^t$. The semi-invariant $f$ is given by $f(M)=\det(M)$ and has degree $n$. We note that $V$ is a multiplicity-free space (cf. \cite{howumed}), i.e. $\complex[V]$ has $G$-irreducible isotypic components. In particular, $f$ has multiplicity-free weight $\sigma=\det^2$.

We have $n+1$ orbits $\orb_0,\orb_1,\dots,\orb_n$ in $V$ under the action of $G$, where $\orb_i$ denotes the set of symmetric matrices of rank $i$. Fix any integer $k$ with $0<k<n$. The defining ideal of $\overline{\orb}_{k-1}$ is generated by the $k\times k$ minors $f_1^{(1)},\dots, f_p^{(1)}$  of the generic symmetric matrix $X$ of variables (for example, see \cite[Theorem 6.3.1]{jerzy}), and these form a basis for an irreducible $G$-submodule $M_\lambda$ of $\complex[V]$, where $\lambda$ is given by the partition $(2^k,0,\dots,0)$. Since $V$ is a multiplicity-free space, the multiplicity one property (\ref{eq:mult1}) holds. We have $\overline{\orb}_{k-1} \subset Z(f)= \overline{\orb}_{n-1}$, so by Example \ref{ex:subvar} we have a (Laplace) expansion of the form (\ref{eq:expand}). By Theorem \ref{thm:mult1}, the $b$-function of $f$ decomposes as $b_n(s)=b_{k,1}(s)\cdot b_{k,2}(s)$, and for any $i=1,\dots,p$ we have the equation
\begin{equation}\label{eq:loc}
\left(\dfrac{1}{f^{(1)}_i (x)} \, f^{*(2)}_i(\partial x) \right) \cdot f^{s+1}(x) = b_{k,2}(s) f^s(x).
\end{equation}
We can choose $f_1^{(1)}$ (resp. $f_1^{*(2)}$) to be $k\times k$ (resp. $(n-k) \times (n-k)$) minor formed by the first $k$ (resp. last $n-k$) rows and columns. We consider the equation (\ref{eq:loc}) with $i=1$, and specialize at 
\[X=\begin{bmatrix}
I_k & 0\\
0 & X_{n-k}
\end{bmatrix},\]
where $X_{n-k}$ is the generic $(n-k)\times (n-k)$ matrix of respective variables. This readily gives the equation for the $b$-function of the symmetric determinant of size $(n-k)\times (n-k)$, hence we obtain $b_{k,2}(s)=b_{n-k}(s)$, and we have the decomposition 
\[b_n(s) = b_{k,1}(s)\cdot b_{n-k}(s).\]
To determine $b_n(s)$ (and a fortiori, all $b_{k,1}(s)$), we consider the case  $k=1$. By Corollary \ref{cor:euler} we have $b_{1,1}(s)=ns + \frac{n(n-1)}{2}$, and we can write (up to a non-zero constant)
\[b_n(s)=\left(s+\frac{n+1}{2}\right) \cdot b_{n-1}(s) =  (s+1)\left(s+\frac{3}{2}\right)\cdots \left(s+\frac{n+1}{2}\right).\]
 % and we can take $v=e_1^2$ as a highest weight vector, with $d=1$. The slice will give $(\GL_{n-1}, 2\lambda_1)$, and condition (\ref{eq:condi}) is satisfied. We have $r=\frac{n+1}{2}$. Hence the $b$-function is $b_n(s)=(s+\frac{n+1}{2})b_{n-1}(s)= (s+1)(s+\frac{3}{2})\cdots (s+\frac{n+1}{2})$.
Now we show that the equations (\ref{eq:loc}) give local $b$-functions at elements in $\orb_k$. Clearly, if $v\in \orb_k$ then there is an $i$ such that $f_i^{(1)}(v)\neq 0$, and (\ref{eq:loc}) shows that the local $b$-function $b_{f,v}(s)$ divides $b_{k,2}(s) = b_{n-k}(s) = (s+1)(s+3/2)\cdots (s+\frac{n-k+1}{2}).$ To see that equality holds, by equivariance we have $b_{f,v} = b_{f,gv}$, for any $g\in G$, which we can denote by $b_{f,\orb_k}$. So it is enough to consider the element $v=\begin{bmatrix} I_k & 0 \\ 0 & 0 \end{bmatrix}$. If we take the slice at $v$ as in Section \ref{subsec:slice}, we get a decomposition $V=\mathfrak{g}v \oplus W$, where we can identify $W$ with the space of $(n-k)\times (n-k)$ symmetric matrices. The induced semi-invariant $f_v$ is the symmetric determinant on $W$. By Lemma \ref{lem:basic}, we have $b_{f,\orb_k}=b_{f,v}(s) = b_{f_v}(s)=b_{n-k}(s)$, hence obtaining the desired equality. We will exploit techniques with slices more systematically in the next section.
\end{example}

\begin{example}\label{ex:ortho} $(\op{SO}(m)\times \GL(n),  \Lambda_1 \otimes \Lambda_1)$, where $m>n$.

This example is also considered in \cite{skko} (although we require only $m>n$). Here $G=\op{SO}(m)\times \GL(n)$, where $\op{SO}(m)$ denotes the special orthogonal group. We think of $V$ as the space of $m\times n$ matrices with the action of $G$ defined by $(h,g) \cdot M = h\cdot M \cdot g^t$, where $h \in \op{SO}(m), g\in \GL(n)$ and $M\in V$. We have a semi-invariant $f_{m,n}$ defined by $f_{m,n}(M) =\det(M^t \cdot M)$ of degree $2n$ and with weight $\sigma = 1 \otimes \det^2$. Since $G$ acts on $V$ with finitely many orbits (see \cite{skko}), $\sigma$ is multiplicity-free. The orthogonal invariants are generated by the entries of $X^t \cdot X$, where $X$ denotes an $m\times n$ generic matrix of variables (see \cite[XI. Section 2.1]{proc}). In fact, this induces a $\GL(n)$-equivariant algebra isomorphism (see \cite[XI. Section 5.2]{proc})
\[\complex[V]^{SO(m)} \cong \complex[\op{Sym}^2 \complex^n].\]
In particular, we have a (Laplace) expansion (\ref{eq:expand}) as in the previous example if we take $M_\lambda$ to be the span of all the $r\times r$ minors $f_1^{(1)},\dots,f_p^{(1)}$ of $X^t\cdot X$  for any $r$ with $1\leq r \leq n-1$, where $\lambda = 1 \otimes (2^r,0,\dots,0)$. Moreover, the above isomorphism shows that the multiplicity one property (\ref{eq:mult1}) holds. By Theorem \ref{thm:mult1} the $b$-function $b_{m,n}(s)$ of $f_{m,n}$ has a decomposition $b_{m,n}(s)= b'_{r,1}(s)\cdot b'_{r,2}(s)$, where $b'_{r,2}(s)$ satisfies the equation
\[\left(\dfrac{1}{f^{(1)}_1 (x)} \, f^{*(2)}_1(\partial x) \right) \cdot f_{m,n}^{s+1}(x) = b'_{r,2}(s) f_{m,n}^s(x).\]
Here $f^{(1)}_1$ (resp. $f^{*(2)}_1$) is the $r\times r$ (resp. $(n-r)\times (n-r)$) minor formed by the first $r$ (resp. last $n-r$) rows and colums of $X^t \cdot X$ (resp. that in dual variables). Specializing the equation above at 
\[X=\begin{bmatrix}
I_r & X_r\\
0 & X_{n-r}
\end{bmatrix},\]
and simplifying $f_{m,n}$, we obtain precisely the equation for the $b$-function of the semi-invariant $f_{m-r,n-r}$ in the variables of $X_{n-r}$. Hence $b'_{r,2}(s)=b_{m-r,n-r}(s)$, and we have a decomposition 
\[b_{m,n}(s) = b'_{r,1}(s) \cdot b_{m-r,n-r}(s).\]
To compute $b_{m,n}(s)$ (hence, a fortiori all $b'_{r,1}(s)$ as well), we choose $r=1$. In this case $f_1^{(1)},\dots,f_p^{(1)}$ are just the entries of $X^t\cdot X$ and $f_1^{*(1)},\dots,f_p^{*(1)}$ the respective dual elements.  By Theorem \ref{thm:mult1}, $b'_{1,1}(s)$ is given by
\[\left[\ds_{i=1}^p f^{*(1)}_i(\partial x) f^{(1)}_i (x)\right] \cdot f_{m,n}^{s}(x) = b'_{1,1}(s) f_{m,n}^s(x).\]
Since this is involves only a 2nd-order differential operator, by a direct computation we obtain (up to constant) that $b'_{1,1}(s) = (s+ \frac{n+1}{2}) (s+ \frac{m}{2})$. Hence we get
\[b_{m,n}(s) = \left(s+ \frac{n+1}{2}\right)\left(s+ \frac{m}{2} \right) \cdot b_{m-1,n-1}(s) = \prod_{i=1}^n \left(s+ \frac{i+1}{2} \right) \left(s+ \frac{m-n+i}{2}\right).\]
\end{example}

\begin{example}\label{ex:symp}  $(\op{Sp}(2m) \times \GL(2n), \Lambda_1 \otimes \Lambda_1)$, where $m>n$.

This example appears also in \cite{kimu} (although we require only $m>n$). Again, we think of $V$ as the space of $2m\times 2n$ matrices. The semi-invariant is the Pfaffian of $f(M)=\op{Pf}(M^t \cdot J \cdot M)$, where $J=\begin{bmatrix} 0 & -I_m \\ I_m & 0 \end{bmatrix}$. The argument is entirely analogous to the previous example, so we omit the details. For each $r$, we obtain a decomposition of the $b$-function of $f$ as $b_{m,n}(s) = b'_{r,1}(s) \cdot b_{m-r,n-r}(s)$. Putting $r=1$, we obtain
\[b_{m,n}(s) =(s+2n-1)(s+2m)\cdot b_{m-1,n-1}(s) = \prod_{i=1}^n \left(s+ 2i-1\right)\left(s+2(m-n+i) \right).\]
\end{example}

\begin{example}\label{ex:cubic} $(\GL(2), 3\Lambda_1)$, the space of binary cubics.

This example appears also in \cite{skko}. Here $V=\op{Sym}^3 \complex^2$ is the space of binary cubic forms with the natural action of $G=\GL(2)$. If we choose $w_0,w_1$ to be a basis of $\complex^2$, then we choose the basis $\{w_0^3, 3w_0^2 w_1, 3w_0 w_1^2, w_1^3\}$ for $V$. Let $x=(x_0, x_1, x_2,x_3)$ the respective coordinate system. The semi-invariant $f\in \complex[V]_4$ is the discriminant 
\[ f = 3x_1^2x_2^2-4x_0x_2^3-4x_1^3x_3-x_0^2x_3^2+6x_0 x_1 x_2 x_3.\]
Since $V$ has only $4$ orbits under the action of $G$, the weight $\sigma = \det^6$ is multiplicity-free. For each $k$ with $0<k<4$, we describe the expansion (\ref{eq:expand}) and show that in each case the multiplicity one property (\ref{eq:mult1}) holds. To this end, we use the $G$-decomposition of $\complex[V]$ described as rational function in \cite[Section 6.1]{series} (we follow the notation as in \cite[Lemma 2.1]{binary})
\begin{equation}\label{eq:binpol}
\op{Sym}(\op{Sym}^3 \complex^2) = \dfrac{1+(6,3)}{(1-(3,0)(1-(4,2))(1-(6,6))},
\end{equation}
where irreducible $G$-modules correspond to pairs of integers $(a,b)$ with $a\geq b$.

When $k=1$, then we have a decomposition (\ref{eq:expand}) with $M_\lambda = \complex[V]_1$ so that $\lambda=(3,0)$ and $\lambda^* \cdot \sigma = (0,-3) + (6,6) = (6,3)$. By (\ref{eq:binpol}) we see that the multiplicity of $\lambda \cdot \sigma^2 = (3,0)+(12,12) = (15,12)$ in $\complex[V]$ is one. Hence (\ref{eq:mult1}) holds, and by Theorem \ref{thm:mult1} we have a decomposition for the $b$-function $b(s)$ of $f$ as $b(s)=b_{1,1}(s) \cdot b_{1,2}(s)$. By Corollary \ref{cor:euler}, we have (up to a constant) $b_{1,1}(s)=s+1$. Since the equation for $b_{1,2}(s)$ involves only a 3rd-order differential operator, one can obtain by a direct calculation that $b_{1,2}(s)=(s+1)(s+5/6)(s+7/6)$.

When $k=2$. we can take $\lambda=(4,2)$ and $\lambda^* \cdot \sigma = (4,2)$. We see from (\ref{eq:binpol}) that the multiplicity of $\lambda \cdot \sigma^1=  (10,8)$ in $\complex[V]$ is one. Hence (\ref{eq:mult1}) holds, and by Theorem \ref{thm:mult1} we have a decomposition $b(s) = b_{2,1}(s) \cdot b_{2,2}(s)$. We give more details for this case. A basis of $M_\lambda=M_{\lambda^*\cdot \sigma}$ (resp. basis of $N_{\lambda^*}=N_{\lambda \cdot \sigma^{-1}}$) is given by the $2\times 2$ minors of
\[\begin{bmatrix}
x_0 & x_1 & x_2\\
x_1 & x_2 & x_3
\end{bmatrix}, \mbox{ resp. }
\begin{bmatrix}
3\partial_0 & \partial_1 & \partial_2\\
\partial_1 & \partial_2 & 3\partial_3
\end{bmatrix}.\]
We choose the basis $\{f_i^{(1)}\}$ and its the dual (up to constant) basis $\{f_i^{*(1)}\}$ with respect to the pairing (\ref{eq:pair}) as follows:
\[\arraycolsep=10pt\begin{array}{ll}
f_1^{(1)} = x_0 x_2 - x_1^2,  & f_1^{*(1)} = 6 \partial_0 \partial_2 - 2\partial_1^2, \\
f_2^{(1)} = x_1 x_3 - x_2^2,  & f_2^{*(1)} = 6\partial_1 \partial_3 - 2 \partial_2^2, \\
f_3^{(1)} = x_0 x_3 - x_1 x_2,  & f_3^{*(1)} = 9 \partial_0 \partial_3 - \partial_1\partial_2.
\end{array}\]
Next, it is easy to see that we can make the choice $f_1^{*(2)}= 3\partial_1 \partial_3 - \partial_2^2$. Now by a direct computation we obtain by Theorem \ref{thm:mult1} that (up to constant) $b_{2,1}(s)=(s+1)(s+5/6)$ and $b_{2,2}(s)=(s+1)(s+7/6)$.

When $k=3$, we have the same expansion for $f$ as with $k=1$, but with the roles of $\lambda$ and $\lambda^*\cdot \sigma$ interchanged. Namely, now $\lambda=(6,3)$ and $\lambda^*\cdot \sigma = (3,0)$. It is easy to see from (\ref{eq:binpol}) that the multiplicity of $\lambda=(6,3)$ in $\complex[V]$ is one, hence (\ref{eq:mult1}) holds. Again, by Theorem \ref{thm:mult1} we have a decomposition $b(s)=b_{3,1}(s)\cdot b_{3,2}(s)$, and it is immediate that $b_{3,2}(s)=s+1$, hence $b_{3,1}=(s+1)(s+5/6)(s+7/6)$.
\end{example}

\begin{example} $(\GL(6), \bigwedge^3 \complex^6)$

This example appears in \cite{skko} and is very similar to the one above, so we omit the details. There exists a semi-invariant $f$ of degree $4$. The $G$-decomposition of $\complex[V]$ is described in \cite[Section 6]{series}. Using this, it is easy to see that the multiplicity one property (\ref{eq:mult1}) holds for all cases $k=1,2,3$, just as in the above example. Hence one can apply Theorem \ref{thm:mult1} here as well and obtain decompositions of the $b$-function of $f$.
\end{example}

\subsection{The slice method}\label{subsec:method}

In general, the multiplicity one property (\ref{eq:mult1}) is not easy to check directly. Several criteria are given in \cite[Section 2]{sasu}, but these are not sufficient for our purposes. Indeed, the authors in \cite{sasu} bring attention to the problem of finding a more satisfactory criterion for the multiplicity one property to hold. Although difficult to answer in general, using slices as in Section \ref{subsec:slice} we derive an efficient criterion that is relatively easy to use. We call this process the slice method.

For the standard theory of reductive groups that we use, we refer the reader to \cite{borel}. Assume $G$ is a connected reductive group, $T$ a maximal torus of $G$ and $B$ a Borel subgroup and $B^-$ an opposite Borel subgroup so that $B\cap B^- = T$.  In this section, $V$ is a rational $G$-module with a $G$-decomposition $V=E\oplus F$ as in Example \ref{ex:main}. We have an algebra isomorphism $\complex[V]=\complex[E]\otimes \complex[F]$. As explained before, for $f\in \op{SI}(G,V)_\sigma$ with $\sigma$ multiplicity-free, we can write 
\begin{equation}\label{eq:decomp}
f(x,y) = \ds_{i=1}^p  f_i^{(1)}(x)f_i^{(2)}(y),
\end{equation}
for $x\in E, y\in F$, where $f_1^{(1)},\dots, f_p^{(1)}$ is a basis for a $G$-irreducible isotypic component $\complex[E]_\lambda$, and $f_1^{(2)},\dots,f_p^{(2)}$ is a $G$-dual basis for the irreducible $\complex[F]_{\sigma \cdot \lambda^*}$, for some dominant weight $\lambda$. 

We can assume WLOG that $f_1^{(1)},\dots, f_p^{(1)}$ is a $T$-weight basis of $\complex[E]_\lambda$ and $f_1^{(1)}$ is the highest weight vector, that is, $f_1^{(1)}$ is a $B$-semi-invariant of weight $\lambda$. Then  $f_1^{(2)},\dots,f_p^{(2)}$ is a $T$-weight basis of $\complex[F]_{\sigma \cdot \lambda^*}$, and $f_1^{(2)}$ is a lowest weight vector, that is, a $B^-$-semi-invariant of weight $\lambda^{-1} \cdot \sigma$. For simplicity, put $f^{(1)}:=f_1^{(1)}$ and $f^{(2)}:=f_1^{(2)}$.

 Let $f^*\in \complex[V^*]$ be the dual of $f$, which is a semi-invariant of multiplicity-free weight $\sigma^{-1}$.  We have an algebra isomorphism $\complex[V^*]=\complex[E^*]\otimes \complex[F^*]$, and we can write
\[f(x^*,y^*) = \ds_{i=1}^p  f_i^{*(1)}(x^*)f_i^{*(2)}(y^*),\]
for $x^*\in E^*,y^*\in F^*$. Here  $f_1^{*(1)},\dots, f_p^{*(1)}$ (resp. $f_1^{*(2)},\dots,f_p^{*(2)}$) is the dual basis of $f_1^{(1)},\dots, f_p^{(1)}$ (resp. $f_1^{(2)},\dots,f_p^{(2)}$) with respect to (\ref{eq:pair}). In particular, $f^{*(2)}:=f^{*(2)}_1$ is a highest weight vector, that is, a $B$-semi-invariant of weight $\lambda \cdot \sigma^{-1}$.

Since $f^{(1)}\in \complex[E]$ is a highest weight vector with dominant weight $\lambda$, the stabilizer of the line $\complex \cdot f^{(1)}$ is a parabolic subgroup $P$ of $G$. Moreover, since $f^{(2)}$ is a lowest weight vector of weight $\sigma\cdot\lambda^{-1}$, the opposite parabolic subgroup $P^-$ is the stabilizer of the line $\complex \cdot f^{(2)}$. We have $P\cap P^-=L$, where $L$ is the Levi subgroup of $P$, which is a connected reductive group.% and $f^{(2)}\in \op{SI}(L,F)$ is a semi-invariant of weight $(\lambda^{-1}\cdot \sigma)|_L$. 

%One can make the above considerations more explicit. Let $(X(T),\Phi,X(T)^\vee,\Phi^\vee)$ be the root datum of $G$ with pairing $\langle \,\, ,\, \rangle : X(T)\times X(T)^\vee \to \integ$. Let $\mathfrak{g}=\mathfrak{n}^-\oplus\mathfrak{t}\oplus\mathfrak{n}^+$ the root space decomposition for the Lie algebra of $G$, with $\mathfrak{n}^{\pm}=\displaystyle\bigoplus_{\alpha\in \Phi^{\pm}} \mathfrak{g_{\alpha}}$, where $\Phi^+$ (resp. $\Phi^-$) is the set of positive (resp. negative) roots, $\Phi=\Phi^- \cup \Phi^+$. Let $\Delta$ be a fixed set of simple roots $\alpha_i$. Put $I=\{\alpha_i \in \Delta | \langle \lambda, \alpha_i^\vee \rangle=0\}$, and let $\Phi_I$ be the root subsystem of $\Phi$ spanned by the simple roots in $I$. Denote $\mathfrak{n}^{\pm}_I=\displaystyle\bigoplus_{\alpha\in \Phi^{\pm}_I} \mathfrak{g}_{\alpha}$. Then $\mathfrak{p}=\mathfrak{n}_I^-\oplus \mathfrak{t} \oplus \mathfrak{n}^+$ is the parabolic subalgebra corresponding to $P$, and $\mathfrak{l}=\mathfrak{n}_I^-\oplus \mathfrak{t} \oplus \mathfrak{n}_I^+$ is the Levi subalgebra corresponding to $L$. %Putting $\mathfrak{u}^\pm=\displaystyle\bigoplus_{\alpha\in \Phi^\pm\backslash \Phi^\pm_I} \mathfrak{g}_{\alpha}$ and $U^\pm$ the corresponding unipotent subgroups, we have the usual Levi decomposition $P= L \ltimes U$.

We assume that we have an element $v\in E$ such that $f^{(1)}(v)=1$, and $f^{(1)}_i(v)=0$, for $i\neq1$. Additionally, we assume that $v$ has a dense $P$-orbit in $E$ (for example, when the action of $G$ on $E$ is multiplicity-free, i.e. $E$ has a dense $B$-orbit -- see \cite{howumed}). With notation from Section \ref{subsec:slice} (choosing $H=P$), we have a decomposition $V=\mathfrak{p}v \oplus F$, and we consider the (slice) representation $(L,F)$ at $v$. Putting $x=v$ in (\ref{eq:decomp}) we get that $f_v = f^{(2)}$ is an $L$-semi-invariant on $F$ of weight $(\lambda^{-1} \cdot \sigma)|_{L}$ (restriction to $L$). If $L_v$ denotes the stabilizer of $v$, then we will see that $f_v$ is an $L_v$-semi-invariant of weight $\sigma|_{L_v}$.

\begin{theorem}\label{thm:maintheor}
Let $v\in E$ as above, assume that the weight $\sigma|_{L_v}$ in $\complex[F]$ is multiplicity-free. Then the multiplicity one property (\ref{eq:mult1}) holds in $\complex[V]$, and we have a decomposition $b_f(s) = b_1(s)\cdot b_2(s)$ as in Theorem \ref{thm:mult1} with $b_2(s)=b_{f_v}(s)=b_{f,v}(s)$.
\end{theorem}

\begin{proof}
First, we show that $\lambda | _{L_v} = 1$. The polynomial $f^{(1)}\in\complex[E]$ is an $L_v$-semi-invariant of weight $\lambda | _{L_v}$. In particular, we have $ (l \cdot f^{(1)})(v) = \lambda(l) f^{(1)}(v) = \lambda(l)$, for any $l \in L_v$. On the other hand, we have $(l \cdot f^{(1)})(v)=f^{(1)}(l^{-1} v) = f^{(1)}(v) =1$, hence  $\lambda | _{L_v} = 1$. This implies that $f_v$ is an $L_v$-semi-invariant of weight $\sigma|_{L_v}$.

Now we show that property (\ref{eq:mult1}) holds. In fact, we prove that the multiplicity in $\complex[V]$ of the irreducible corresponding to $\lambda \cdot \sigma^k$ is one, for any $k\in\nat$. As noted in Remark \ref{rem:works}, the considerations in Section \ref{subsec:slice} work for the slice representation $(L_v,F)$ with the group $L_v$ (although this group is defined in a different way than the one defined in that section). Using that $\lambda | _{L_v} = 1$, the map (\ref{eq:slice}) in this case is
\[\phi_v^{\lambda\cdot\sigma^k} : \op{SI}(P,V)_{\lambda\cdot\sigma^k} \to \op{SI}(L_v,F)_{\sigma^k |_{L_v}}.\]
Since the weight of $\sigma|_{L_v}$ is multiplicity-free,  the space $\op{SI}(L_v,F)_{\sigma^k |_{L_v}} = \complex \cdot f_v^k$ is one-dimensional. By Lemma \ref{lem:basic}, $\phi_v^{\lambda\cdot\sigma^k}$ is injective, hence $\op{SI}(P,V)_{\lambda\cdot\sigma^k}=\complex \cdot (f^{(1)}f^k)$. This implies that the multiplicity one property (\ref{eq:mult1}) holds. By Theorem \ref{thm:mult1} (2)  we have an equation
\begin{equation}\label{eq:eval}
f^{*(2)}(\partial y) \cdot f^{s+1}(x,y) = b_2(s) f^{(1)}(x) f^s(x,y).
\end{equation}
Since $\sigma|_{L_v}$ is a multiplicity-free weight and $\lambda | _{L_v} = 1$, the $L$-semi-invariant $f_v$ has multiplicity-free weight $(\lambda^{-1} \cdot \sigma)|_{L}$ for the reductive group $L$ (in particular, $f$ is homogeneous). Recall that $f^{*(2)}$ has highest weight $\lambda\cdot \sigma^{-1}$, hence it is an $L$-semi-invariant of weight $(\lambda\cdot \sigma^{-1})|_L$. This shows that (up to constant)  $f^{*(2)}$ is the dual $L$-semi-invariant of $f_v$ on $F$, that is $f_v^* = f^{*(2)}$. Now specializing at $x=v$ in the equation (\ref{eq:eval}) we obtain
\[f^*_v (\partial y) \cdot f_v^{s+1} (y) = b_2(s) f^s_v(y).\]
By Theorem \ref{thm:bfun} this equation gives precisely the $b$-function of $f_v$, so $b_2(s)=b_{f_v}(s)$. To see that equation (\ref{eq:eval}) gives indeed the local $b$-function of $f$ at $v$, we use Lemma \ref{lem:basic} again and obtain $b_{f_v}(s)=b_{f,v}(s)$.
\end{proof}

In the case of multiplicity-free tuples as in Definition \ref{def:tuple}, we can say more about the first factor $b_1(s)$ in Theorem \ref{thm:maintheor}:

\begin{lemma}\label{lem:firstfac}
Assume additionally, that $\ul{f}=(f_1^{(1)},\dots, f_p^{(1)})$ from (\ref{eq:decomp}) is a multiplicity-free tuple. Then $b_1(s)=P_{\ul{f}}(s)$ , with $P_{\ul{f}}(s)$ as in Proposition \ref{prop:pf}. In particular, the Bernstein-Sato polynomial $b_{\ul{f}}(s)$ divides $b_1(s)$.
\end{lemma}

\begin{proof}
It is enought to show $b_1(s)=P_{\ul{f}}(s)$ for an arbitrary positive integer $s$. Denote by $M_s$ the irreducible $G$-module as in Definition \ref{def:tuple}. By Theorem \ref{thm:mult1}, $b_1(s)$ is given by the equation
\[\left[\ds_{i=1}^p f^{*(1)}_i(\partial x) f^{(1)}_i (x)\right] \cdot f^{s}(x,y) = b_1(s) f^s(x,y).\]
We can evaluate the equation at any point $y=w\in F$. Choose $w\in F$ such that the polynomial $f(x,w)\in \complex[E]$ is not zero. By the expansion (\ref{eq:decomp}), we see that $f^s(x,w) \in M_s$, for any $w\in F$. By Schur's Lemma and Proposition \ref{prop:pf}, $D_{\ul{f}}$ acts on $M_s$ by the scalar $P_{\ul{f}}(s)$, which then coincides with $b_1(s)$ by the equation above.
\end{proof}

Now we formulate a result for the important case when the representation $(G,V)$ is of the form 
\begin{equation}\label{eq:matrix}
(\GL(m)\times \GL(n) \times G'\,,\, \Lambda_1^{(*)} \otimes \Lambda_1^{(*)} \otimes 1 \oplus 1 \otimes \rho\,,\, M_{m,n}\oplus F)
\end{equation}
with $m\leq n$. This is the main case considered also in \cite{sasu} and \cite{sugi}, and we use the notation as in \cite[Section 2.1]{sasu}. Namely, here $G'$ is an arbitary connected reductive group, $ \rho$ is an arbitary rational representation of $\GL(n) \times G'$, and  $\Lambda_1^{(*)}$ is either the standard representation of $\GL$ or its dual (for simplicity, we take WLOG the duals $\Lambda_1^{*}$).  Many prehomogeneous vector spaces are of this form -- see Sections \ref{sec:quiv} and the classification in \cite{saki}.

We define
\[H=\GL(m) \times \GL(n-m) \times G'\]
to be the the reductive subgroup of $\GL(n) \times G' \subset  G$ , with the factor $\GL(m) \times \GL(n-m)$ of $H$ embeds into $\GL(n)$ as

\[\GL(m) \times \GL(n-m) = \left\{\begin{bmatrix} A & 0 \\ 0 & B \end{bmatrix} : \mbox{ where } A\in \GL(m), B\in \GL(n-m) \right\}\subset \GL(n).\]

 Let $I \subset \complex[M_{m,n}]$ denote the ideal generated by the maximal minors as introduced in Section \ref{subsec:bideal}. Choose $v=\left(\begin{bmatrix}I_m & 0_{n-m} \end{bmatrix},0\right) \in  M_{m,n}\oplus F$.

\begin{theorem} \label{thm:mainapp}
Consider the space $V=M_{m,n}\oplus F$ as in (\ref{eq:matrix}) and let $f\in \complex[V]$ be a $G$-semi-invariant of weight $\sigma=\det^d \otimes \det^e \otimes \sigma'$, where $d,e\in \nat$ and $\sigma'$ is a character of $G'$. Assume that $\det^{e-d} \otimes \det^e \otimes \sigma'$ is a multiplicity-free character of $H$ in $\complex[F]$. Then $\sigma$ is a multiplicity-free character of $G$ in $\complex[V]$ and the multiplicity one property (\ref{eq:mult1}) holds. Moreover, the Bernstein-Sato polynomial of $f$ decomposes as $b_f(s)=b_1(s) \cdot b_2(s)$ where:
\begin{itemize}
\item[(1)] $b_1(s) =b_{I^{d}}(s) = [s]_{n-m,n}^{d}$ is the Bernstein-Sato polynomial of the ideal $I^{d}$;
\item[(2)] $b_2(s) = b_{f,v}(s) = b_{f_v}(s)$ is the Bernstein-Sato polynomial of the induced semi-invariant $f_v$ on the slice $(H,F)$, which is also equal to the local Bernstein-Sato polynomial of $f$ at $v$.
\end{itemize}
\end{theorem}

\begin{proof}
The stabilizer $G_v$ of $v$ is formed by all elements of the form
\[\left(A^{-1}, \begin{bmatrix} A  & 0 \\ C & B \end{bmatrix}, g\right) \subset \GL(n) \times \GL(m)\times G'.\]
Let $L_v$ be the reductive subgroup of $G_v$ formed by the elements as above with $C=0$. Clearly, $L_v$ is isomorphic to $H$ (by forgetting the first factor).

We have $\mathfrak{g}v=M_{m,n}$. As in Section \ref{subsec:slice}, we consider the map from $\complex[V]$ to $\complex[F]$ given by $h \mapsto h_v$, where $h\in \complex[V]$, and $h_v\in\complex[F]$ is defined by $h_v(y)=h(v+y)$, for $y\in F$. Fix $k\in\nat$ and assume $h\in \complex[V]$ is a $G$-semi-invariant of weight $\sigma^k$. As seen in Section \ref{subsec:slice}, $h_v\in\complex[F]$ is then an $L_v$-semi-invariant of weight $\sigma^k|_{L_v}=(\det^{e-d} \otimes \det^e \otimes \sigma')^k$. Since the first factor $\GL(n)$ of $G$ acts on $F$ trivially, in fact $h_v$ is also an $H$-semi-invariant of weight $(\det^{e-d} \otimes \det^e \otimes \sigma')^k$. This shows that we have a map as (\ref{eq:slice}):
\[\phi_v^{\sigma^k} : \op{SI}(G,V)_{\sigma^k} \to \op{SI}(H,F)_{(\det^{e-d} \otimes \det^e \otimes \sigma')^k}.\]
Since $\det^{e-d} \otimes \det^e \otimes \sigma'$ is multiplicity-free, $ \op{SI}(H,F)_{(\det^{e-d} \otimes \det^e \otimes \sigma')^k}=\complex \cdot f_v^k$ is one-dimensional. By Lemma \ref{lem:basic} (taking into account Remark \ref{rem:works}) the map $\phi_v^{\sigma^k}$ is injective, hence $\op{SI}(G,V)_{\sigma^k}=\complex \cdot f^k$ and $\sigma$ is multiplicity-free. In particular, we have an expansion of the form (\ref{eq:decomp}). By FFT  (see \cite[XI. Section 1.2]{proc}), the elements $f^{(1)}_1,\dots, f^{(1)}_p$ are a basis of the irreducible $\GL(m)\times\GL(n)$-module $M_\lambda$ of $\complex[M_{m,n}]$, where the dominant weight is $\lambda = \det^d \otimes (d^m, 0^{n-m}) \otimes 1$ (see also \cite[Section 2.1]{sasu}). We choose $f^{(1)}_1,\dots, f^{(1)}_p$ to be elements that are products of $d$ maximal minors. We can take $f_1^{(1)}$ to be the $d$th power of the maximal minor corresponding to the first $m$ columns, which, by a standard choice of a Borel subgroup $B$ of $\GL(m)\times\GL(n)$, is highest weight vector. Note that under this choice the $B$-orbit of $v$ is dense in $M_{m,n}$. Also, $f_1^{(1)}(v)=1$,  while $f_i^{(1)}(v) = 0$, for $i\neq 1$, so $f^{(2)}_1 = f_v$ (see the considerations before Theorem \ref{thm:maintheor}). Let $P$ be the parabolic subgroup of $G$ corresponding to $\lambda$, i.e. the stabilizer of the line $\complex \cdot f_1^{(1)}$, and let $L$ be the corresponding Levi subgroup. Then it is easy to see that the stabilizer of $v$ in $L$ is the same as the group $L_v$ constructed above. Since  $\sigma'|_H$ is multiplicity-free, $\sigma|_{L_v}$ is multiplicity-free on $F$ as well. 

We showed that all the assumptions in Theorem \ref{thm:maintheor} are satisfied. This, together with Lemma \ref{lem:firstfac} and Theorem \ref{thm:powermax}, yields the conclusion.
\end{proof}

The technique can be used to determine an explicit representative for the locally semi-simple point of $f$ (see Section \ref{subsec:bfun}).

\begin{proposition}\label{prop:locsemi}
Consider a semi-invariant $f\in \complex[V]$ as in Theorem \ref{thm:mainapp} with $d\neq 0$. Let $w\in F$ be the locally semi-simple point of $f_v \in \op{SI}(H,F)$. Then $v+w$ is the locally semi-simple point of $f\in \op{SI}(G,V)$.
\end{proposition}

\begin{proof}
Take any $z \in V$ such that $f(z)\neq 0$. We want to show that $v+w \in \overline{Gz}$. Since $d\neq 0$, the orbit $G\cdot z$ has an element the form $v+w'$, where $w' \in F$. Since $f_v(w') = f(v+w') \neq 0$ and $w$ is the locally semi-simple point of $f_v$, we must have that $w \in \overline{H\cdot w'}=\overline{L_v \cdot w'}$. Since $L_v$ fixes $v$, we have $v+w \in \overline{L_v \cdot (v+w')}$. This shows that $v+w$ is in the closure of the $G$-orbit of $z$.
\end{proof}

As seen in the proof of Theorem \ref{thm:mainapp}, the stabilizer $G_v$ of $v$ decomposes as a semi-direct product $G_v = L_v \ltimes U$, where $U\cong \Hom_\complex(\complex^m, \complex^{n-m})$ is a unipotent subgroup.

\begin{proposition}\label{prop:algiso}
Consider the space $V=M_{m,n} \oplus F$ as in (\ref{eq:matrix}). Then the map $\phi_v$ from (\ref{eq:slice}) induces an isomorphism of algebras
\[ \phi_v : \op{SI}(G,V) \cong \op{SI}(H,F)^{U}.\]
\end{proposition}

\begin{proof}
By Lemma \ref{lem:basic}, $\phi_v$ is injective on the level of weight spaces. A $G$-semi-invariant of weight $\det^d \otimes \det^e \otimes \sigma'$ is mapped to an $H$-semi-invariant of weight  $\det^{e-d} \otimes \det^e \otimes \sigma'$. Since $m<n$, this shows that different weight spaces are mapped to different weights spaces, so $\phi_v$ is injective.

Now we show that $\phi_v$ is surjective. Let $f' \in \op{SI}(H,F)^U$ be an $H$-semi-invariant of weight $\det^{a} \otimes \det^b \otimes \sigma'$, for some $a,b\in \integ$ and character $\sigma'$ of $G'$. Consider the character $\sigma$ of $G$ defined as $\sigma = \det^{b-a} \otimes \det^b \otimes \sigma'$. Consider the function $F$ defined on the open set $G\cdot v \times F \subset V$ by
\[ F(g\cdot v, y) = \sigma(g) \cdot f'(g^{-1} \cdot y), \mbox{ for } g\in G, y\in F.\]
Using that $f'$ is $G_v$-semi-invariant, we see that $F$ is a well-defined semi-invariant of weight $\sigma$. Since $m<n$, the open set $G\cdot v \times F$ has codimension $\geq 2$ in $V$. Hence $F$ extends to a global semi-invariant, and $F_v = f'$.
\end{proof}

\begin{remark}
We note that the results above regarding $b$-functions hold for the case $m=n$ in (\ref{eq:matrix}) as well. Moreover, in this case there is an algebra isomorphism analogous to Proposition \ref{prop:algiso}
\[\phi_v: \op{SI}(G,V) / (\det X - 1) \cong \op{SI}(H,F),\]
where $X$ is the generic matrix of variables on $M_{n,n}$. For results in this direction obtained by slicing at elements other then our choice $v$, cf. \cite[Section 4]{phd}.
\end{remark}

We conclude the section by mentioning that most results for $b$-functions of one variable can be extended readily to the case of $b$-functions of \textit{several} variables as in Lemma \ref{lem:several}. We will mention only the extension of Theorem \ref{thm:mainapp} to this case, the proof of which is analogous, \textit{mutatis mutandis}. 

\begin{theorem}\label{thm:multi}
Consider the space $V=M_{m,n}\oplus F$ as in (\ref{eq:matrix}), and let $\ul{f}=(f_1,\dots, f_l)$ be $G$-semi-invariants in $\complex[V]$ of weights $\sigma_1,\dots, \sigma_l$, respectively, where $\sigma_i = \det^{d_i} \otimes \det^{e_i} \otimes \sigma'_i$, for $i=1,\dots,l$, with $d_i,e_i\in \nat$ and $\sigma'_i$ a character of $G'$.  Assume the product $\prod_{i=1}^l \det^{e_i-d_i} \otimes \det^{e_i} \otimes \sigma'_i$ is a multiplicity-free character of $H$ in $\complex[F]$. Then the product $\sigma_1 \cdots \sigma_l$ is a multiplicity-free character of $G$ in $\complex[V]$. Moreover, the $b$-function of several variables decomposes as 
\[b_{\ul{f},\ul{m}}(\ul{s}) = [s]_{n-m,n}^{d_1,\dots,d_l} \cdot  b_{\ul{f}_v,\ul{m}}(\ul{s}),\] 
for any tuple $\ul{m}=(m_1,\dots,m_l)\in \nat^l$, where $b_{\ul{f}_v,\ul{m}}(\ul{s})$ is the $b$-function of several variables of the tuple  $\ul{f}_v=(f_{1,v},\dots, f_{l,v})$ of induced semi-invariants on the slice $(H,F)$.

%$$b_{\underline{m}}(\underline{s})=b_{v,\underline{m}}(\underline{s})\prod_{i=0}^{d-1} (d_1 s_1 + \dots + d_l s_l+r+i),$$
%where $d=\ds_{i=1}^l d_i m_i$.
\end{theorem}

\section{Semi-invariants of quivers and the slice method}
\label{sec:quiv}

In this section we apply the methods the slice method from Section \ref{subsec:method} to semi-invariants of quivers.

\subsection{Background on quivers and their semi-invariants}\label{subsec:backquiv}

In this section we will introduce some basics of quivers and semi-invariants. For more background material, we refer the reader to \cite{elements,harwey}. We follow similar notation to that in \cite{en}.

A quiver $Q$ is an oriented graph, i.e. a pair $Q=(Q_0,Q_1)$ formed by the set of vertices $Q_0$ and the set of arrows $Q_1$. An arrow $a$ has a head $ha$, and tail $ta$, that are elements in $Q_0$:
\[\xymatrix{
ta \ar[r]^{a} & ha
}\]
%A vertex $x\in Q_0$ is called a sink (resp. source) if there is no arrow in $Q$ starting (resp. ending) in $x$.

We assume in throughout that $Q$ is a quiver \itshape without oriented cycles \normalfont.

A representation $V$ of $Q$ is a family of finite dimensional vector spaces $\{V(x)\,|\, x\in Q_0\}$ together with linear maps $\{V(a) : V(ta)\to V(ha)\, | \, a\in Q_1\}$. The dimension vector $\underline{d} (V)\in \nat^{Q_0}$ of a representation $V$ is the tuple $\Dim V:=(\dim V(x))_{x\in Q_0}$. A morphism $\phi:V\to W$ of two representations is a collection of linear maps $\phi = \{\phi(x) : V(x) \to W(x)\,| \,x\in Q_0\}$, with the property that for each $a\in Q_1$ we have $\phi(ha)V(a)=W(a)\phi(ta)$. Denote by $\Hom_Q(V,W)$ the vector space of morphisms of representations from $V$ to $W$. For two vectors $\alpha, \beta\in \integ^{Q_0}$, we define the Euler product
$$\langle \alpha, \beta \rangle = \ds_{x\in Q_0} \alpha_x \beta_x - \ds_{a\in Q_1} \alpha_{ta} \beta_{ha}.$$
Let $E$ denote the Euler matrix corresponding to the Euler product. Then $C=-E^{-1} \cdot E^t$ is the \textit{Coxeter transformation} of $Q$ (see \cite{elements}).

We define the vector space of representations with dimension vector $\alpha\in \nat^{Q_0}$ by
$$\Rep(Q,\alpha):=\displaystyle\bigoplus_{a\in Q_1} \Hom(\complex^{\alpha_{ta}},\complex^{\alpha_{ha}}).$$
The group 
$$\GL(\alpha):= \prod_{x\in Q_0} \GL(\alpha_x)$$
acts on $\Rep(Q,\alpha)$ in a natural way by changing basis at each vertex. Under this action, two representations lie in the same orbit if and only if they are isomorphic representations.

For any two representations $V$ and $W$, we have the following exact sequence:
\begin{equation}\label{eq:ringel}
\begin{array}{rlc}
0 \to \Hom_Q (V,W) \stackrel{i}{\longrightarrow} \displaystyle\bigoplus_{x \in Q_0}& \!\!\!\!\!\Hom(V(x),W(x)) & \\
& \stackrel{d^V_W}{\longrightarrow}  \displaystyle\bigoplus_{a\in Q_1} \Hom(V(ta),W(ha)) \stackrel{p}{\longrightarrow} \Ext_Q(V,W)\to 0 &
\end{array}
\end{equation}
Here, the map $i$ is the inclusion, $d_W^V$ is given by
$$\{\phi(x)\}_{x\in Q_0} \mapsto \{\phi(ha)V(a) - W(a)\phi(ta)\}_{a\in Q_1}$$
and the map $p$ builds an extension of $V$ and $W$ by adding the maps $V(ta)\to W(ha)$ to the direct sum $V\oplus W$. From the exact sequence (\ref{eq:ringel}) we have that 
\[\langle \Dim V,\Dim W \rangle = \dim \Hom_Q(V,W) - \dim \Ext_Q(V,W).\]
The orbit $\mathcal{O}_V$ is dense in $\Rep(Q,\alpha)$ if and only if $\Ext_Q(V,V)=0$, in which case we say that $V$ is a \textit{generic} representation, and $\alpha$ a \textit{prehomogeneous dimension vector}.

Now we turn to semi-invariants of a quiver representation space $\Rep(Q,\beta)$. As in Section \ref{sec:b-func}, form the ring of semi-invariants $\op{SI}(Q,\beta)\subset \complex[\Rep(Q,\beta)]$ by 
\[\op{SI}(Q,\beta)=\bigoplus_\sigma \op{SI}(Q,\beta)_\sigma.\]
Here $\sigma$ runs through all the characters of $\GL(\beta)$. Each character $\sigma$ of $\GL(\beta)$ is a product of determinants, that is, of the form 
$\prod_{x\in Q_0} \op{det}_x^{\sigma(x)},$ where $\op{det}_x$ is the determinant function on $\GL(\beta_x)$. In this way, we will view a character $\sigma$ as a function $\sigma : Q_0 \to \integ$, or equivalently, as an element $\sigma\in \Hom_{\integ}(\integ^{Q_0},\integ)$. With this convention, we view characters as duals to dimension vectors, namely:
$$\sigma(\beta)=\ds_{x\in Q_0} \sigma(x)\beta_x.$$

%Following \cite{shme}, for a semi-invariant $f\in \op{SI}(Q,\beta)$, we say that a representation $A\in \Rep(Q,\beta)$ is \itshape locally semi-simple \normalfont (corresponding to $f$), if the orbit $\mathcal{O}_A$ of $A$ is closed in the principal open set defined by $f\neq 0$. Such an orbit is unique if and only if the weight of $f$ is multiplicity-free. In other words, a locally semi-simple representation corresponding to $f$ is a minimal representation $A$ (with respect to containment of orbit closures) for which $f(A)\neq 0$. 

We recall the definition of an important class of determinantal semi-invariants, first constructed by Schofield in \cite{scho}. Fix two dimension vectors $\alpha,\beta$, such that $\langle \alpha, \beta \rangle=0$. The latter condition says that for every $V\in \Rep(Q,\alpha)$ and $W\in \Rep(Q,\beta)$ the matrix of the map $d^V_W$ in (\ref{eq:ringel}) will be a square matrix. We define the semi-invariant $c$ of the action of $\GL(\alpha)\times\GL(\beta)$ on $\Rep(Q,\alpha)\times \Rep(Q,\beta)$ by $c(V,W):=\det d^V_W$. Note that we have
$$c(V,W)=0 \Leftrightarrow \Hom(V,W)\neq 0 \Leftrightarrow \Ext(V,W)\neq 0.$$
Next, for a fixed $V$, restricting $c$ to $\{V\} \times \Rep(Q,\beta)$ defines a semi-invariant $c^V\in \op{SI}(Q,\beta)$. Similarly, for a fixed $W$, restricting $c$ to $\Rep(Q,\alpha)\times \{W\}$, we get a semi-invariant $c_W\in \op{SI}(Q,\alpha)$. The weight of $c^V$ is $\langle \alpha, \cdot \rangle \in \Hom_{\integ}(\integ^{Q_0},\integ)$, and the weight of $c_W$ is $-\langle \cdot, \beta \rangle$. The semi-invariants $c^V$ and $c_W$ are well-defined up to scalar, that is, if $V$ is isomorphic to $V'$, then $c^V$ and $c^{V'}$ are equal up to a scalar. %Furthermore, taking any projective (resp. injective) resolution (not the canonical one that gives (\ref{eq:ringel}) doesn't affect the outcome up to a scalar. We will always ignore such scalars.

\begin{theorem}[{\cite{harwey,schovan}}]\label{thm:span}
For a fixed dimension vector $\beta$, the ring of semi-invariants $\op{SI}(Q,\beta)$ is spanned by the semi-invariants $c^V$, with $\langle \Dim V,\beta \rangle=0$. The analogous result holds for the semi-invariants $c_W$.
\end{theorem}

By \cite[Lemma 1]{harwey}, the algebra of semi-invariants $\op{SI}(Q,\beta)$ is generated by semi-invariants $c^V$, with $\langle \Dim V, \cdot \rangle=0$ and $V$ a Schur representation (that is, $\End_Q(V)=\complex$). We call a prehomogeneous dimension vector $\alpha$ a \itshape real Schur root\normalfont, if the generic representation $V\in\Rep(Q,\alpha)$  is a Schur representation. Note that in this case we have $\langle \alpha, \alpha \rangle=1$. Examples of real Schur roots include the dimension vectors of preprojective and preinjective representations (see \cite{en}).

In the case $\beta$ is a prehomogeneous dimension vector, $\op{SI}(Q,\beta)$ a polynomial ring generated by semi-invariants $c^{V_i}$, where $V_i$ are the simple objects in an appropriate perpendicular category (see \cite[Theorem 4.3]{scho}).

To find semi-invaraints with multiplicity-free weights on spaces $\Rep(Q,\beta)$  with $\beta$ not necessarily prehomogeneous, the following reciprocity result is useful:

\begin{theorem}[{\cite[Corollary 1]{harwey}}]
\label{thm:recip}
Let $\alpha$ and $\beta$ be two dimension vectors, with $\langle \alpha,\beta \rangle=0$. Then
$$\dim \op{SI}(Q,\beta)_{\langle \alpha, \cdot \rangle} = \dim \op{SI}(Q,\alpha)_{-\langle \cdot , \beta \rangle}.$$
\end{theorem}

In particular, if $f$ is a non-zero semi-invariant of weight $\langle \alpha , \cdot \rangle$, with $\alpha$ prehomogeneous, then $f=c^V \in \op{SI}(Q,\beta)$ has multiplicity-free weight, where $V$ is the generic representation in $\Rep(Q,\alpha)$. 

\begin{remark}\label{rem:fulton} 
By the proof of the Generalized Fulton Conjecture (see \cite[Theorem 2.22]{combi}), in order to show that a character $\sigma$ is multiplicity-free in $\op{SI}(Q,\beta)$, it is enough to show that $\dim \op{SI}(Q,\beta)_\sigma = 1$ (i.e. one does not need to check this for higher powers of $\sigma$).
\end{remark}

One can write down the semi-invariants $c^V$ explicitly as determinants of suitable block matrices (see \cite[Remark 3.3]{en}).

\begin{example}
\label{eq:deefor}
Let $Q$ be the following $\D_4$ quiver:
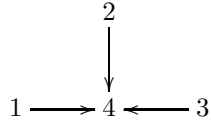
\[\xymatrix{
 & 2 \ar[d] & \\
1 \ar[r] & 4 & \ar[l] 3
}\]
Let $V$ be the indecomposable $V=\begin{matrix} & 1 & \\ 1 & 1 & 1\end{matrix}$. Then $\langle \alpha, \beta \rangle=0$ gives $\beta=(\beta_1,\beta_2,\beta_3,\beta_4)$ with $\beta_1+\beta_2+\beta_3=2\beta_4$. Let $X,Y,Z$ be generic matrices of variables, with $X\in M_{\beta_4,\beta_1},Y\in M_{\beta_4,\beta_2},Z\in M_{\beta_4,\beta_3}$. Then $c^V$ is the determinant of the following square matrix of variables:
$$\det\begin{pmatrix}
X & 0 & 0 & I_{\beta_4}\\
0 & Y & 0 & I_{\beta_4}\\
0 & 0 & Z & I_{\beta_4}
\end{pmatrix}=\det\begin{pmatrix}
X & Y & 0 \\
0 & Y & Z
\end{pmatrix}.$$
Also, $c^V\neq 0$ if and only if $\beta_i\leq \beta_4$, for $i=1,2,3$, and $c^V$ is irreducible if and only if all these inequalities are strict.
\end{example}

In general, slicing a quiver results in a more complicated quiver. However, in some cases we can view a semi-invariant of a quiver as a function on a simpler quiver.
 
 \begin{lemma}[{\cite[Lemma 3.4]{en}}]\label{lem:simp}
Let $Q$ be a quiver without oriented cycles, $\beta$ a dimension vector and $f$ a semi-invariant on $\Rep(Q,\beta)$ of weight $\sigma=\langle \alpha , \cdot \rangle$. Then we can view $f$ as a semi-invariant on a new quiver with new weight according to the following simplification rules:
\begin{itemize}
\item[(a)] If $\alpha_1=0$, then we have (we put the values of $\alpha$ on top of $\beta$):
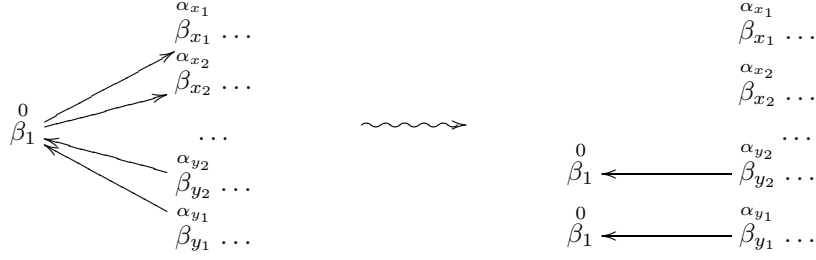
 \[\vcenter{\vbox{\xymatrix@R-2.3pc@C+1.8pc{
 & \stackrel{\alpha_{x_1}}{\beta_{x_1}} \dots \\
 & \stackrel{\alpha_{x_2}}{\beta_{x_2}} \dots \\
\stackrel{0}{\beta_{1}} \ar[uur] \ar[ur] & \dots \\
 & \stackrel{\alpha_{y_2}}{\beta_{y_2}} \ar[ul] \dots \\
 & \stackrel{\alpha_{y_1}}{\beta_{y_1}} \ar[uul] \dots
}}}\hspace{0.4in}{\xymatrix@C+1.5pc{ \ar@{~>}[r] & }}\hspace{0.4in}
 \vcenter{\vbox{\xymatrix@R-2pc@C+2pc{
 & \stackrel{\alpha_{x_1}}{\beta_{x_1}} \dots \\
 & \stackrel{\alpha_{x_2}}{\beta_{x_2}} \dots \\
 & \hspace{0.2in} \dots \\
\stackrel{0}{\beta_{1}} & \stackrel{\alpha_{y_2}}{\beta_{y_2}} \ar[l] \dots \\
\stackrel{0}{\beta_{1}} & \stackrel{\alpha_{y_1}}{\beta_{y_1}} \ar[l] \dots
}}} \]
\item[(b)] Write $\sigma=-\langle \cdot, \alpha^* \rangle$. If $\alpha^*_1=0$, then the same simplification rule holds as in part (a) by replacing $\alpha$ with $\alpha^*$, with the arrows reversed.
\end{itemize}
\end{lemma}
 
If we write $\sigma=\langle \alpha, \cdot \rangle = -\langle \cdot, \alpha^* \rangle$, then $\alpha^* = C\cdot \alpha$, where $C$ denotes the Coxeter transformation. This transformation can understood as applying reflections to sinks successively once at each vertex of the quiver (see \cite{elements},\cite{en}). In particular, if vertex $1$ in part (b) of the above lemma is a sink, then $\alpha_1^*=0$ is equivalent to $\alpha_1 = \ds_{i} \alpha_{x_i}$, in which case one can simply delete vertex $1$.
 
\subsection{The slice method for quivers}\label{subsec:slicequiv} 
 
%In this section we apply results from Section \ref{subsec:method} to compute roots of the b-functions of some quiver semi-invariants.

%For a vertex $x\in Q_0$, denote by $\epsilon_x$ the function on $Q_0$ defined by $\epsilon_x(x)=1$, and $\epsilon_x(y)=0$, for $x\neq y$. Denote by $S_x$ the simple representation of $Q$ corresponding to a vertex $x\in Q_0$, that is, the representation with all linear maps $0$ and $S_x(y)=\complex^{\epsilon_x(y)}$. 

Let $Q$ be a quiver without oriented cycles. We say that an that arrow $a\in Q_1$ is a \itshape 1-source \normalfont (resp. \itshape 1-sink\normalfont) if $ta$ (resp. $ha$) is not a vertex of any arrow other than $a$. We will slice at such arrows $a$ as in Section \ref{subsec:method}. The following is an immediate consequence of Theorem \ref{thm:mainapp}:

\begin{theorem}\label{thm:bquiv}
Let $Q$ be a quiver, $\beta$ be a dimension vector. Let $\vec{a}\in Q_1$ a 1-source or 1-sink arrow, number its vertices by $1,2$, and assume $\beta_1 \leq \beta_2$.  The slice at the arrow $\vec{a}$ is a representation space $\Rep(Q_a,\beta_a)$ corresponding to the following quiver (where the orientation of $\vec{a}$ is arbitrary):
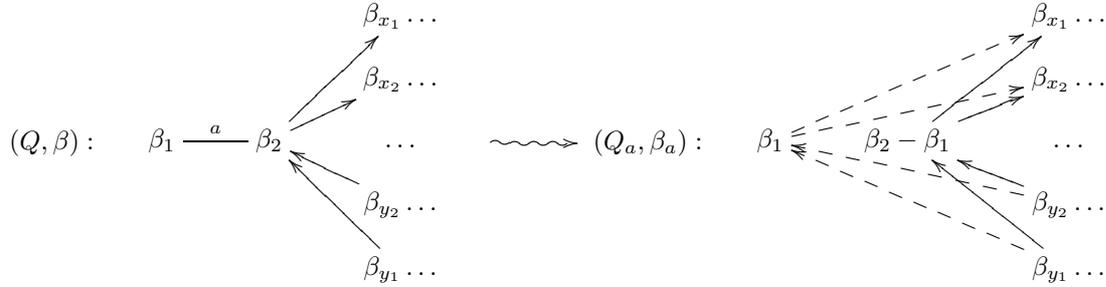
\[(Q,\beta): \hspace{0.2in} \vcenter{\vbox{\xymatrix@R-1.3pc{
 & & \beta_{x_1}  \dots\\
 & & \beta_{x_2}  \dots\\
\beta_{1}  \ar@{-}[r]^a & \beta_{2} \ar[uur] \ar[ur] & \dots\\
 & & \beta_{y_2} \ar[ul] \dots\\
 & & \beta_{y_1} \ar[uul] \dots
}}}\hspace{0.1in}\xymatrix@C+1pc{ \ar@{~>}[r] & }
(Q_a,\beta_a): \hspace{0.2in} \vcenter{\vbox{\xymatrix@R-1.3pc{
 & & \beta_{x_1}  \dots \\
 & & \beta_{x_2}  \dots \\
\beta_{1}  \ar@{-->}[uurr] \ar@{-->}[urr] & \beta_{2}-\beta_{1} \ar[uur] \ar[ur] &\dots \\
 & & \beta_{y_2} \ar[ul] \ar@{-->}[ull] \dots\\
 & & \beta_{y_1} \ar[uul] \ar@{-->}[uull] \dots
}}}\]

Let $f$ be a semi-invariant on $\Rep(Q,\beta)$ of weight $\sigma=\langle \alpha,\cdot \rangle$ and $f_a$ be the induced semi-invariant on $\Rep(Q_a,\beta_a)$ with induced weight $\sigma_a= \langle \alpha_a ,\cdot \rangle$. Under the natural correspondence of vertices between $Q$ and $Q_a$, $\sigma_a$ differs from $\sigma$ only at vertex $1$, with $\sigma_a(1)=\sigma(1)+\sigma(2)$. Moreover, if $\sigma_a$ is a multiplicity-free weight on $\Rep(Q_a,\beta_a)$, then $\sigma$ is multiplicity-free as well and we have
$$b_f(s)=b_{f_a}(s)\cdot [s]^{|\sigma_1|}_{\beta_2-\beta_1,\beta_2}.$$
\end{theorem}

\begin{remark}\label{rem:impo}
In examples, we prefer working with dimension vectors $\alpha$ rather than the weights $\sigma$. Since we know the weight $\sigma_a=\langle \alpha_a, \cdot \rangle$ on the slice, we implicitly also know the dimension vector $\alpha_a$.  Let $P_i$ be the indecomposable projective module (see \cite{elements}) of $Q_a$ at vertex $i$ and $S_i$ the simple module of $Q_a$ at vertex $i$. The formulas are:
\begin{itemize}
\item[(a)] If $a$ is a 1-source, then $\alpha_a=\alpha+(\alpha_2-\alpha_1)\Dim P_1-\alpha_1 \cdot \Dim S_2,$
\item[(b)] If $a$ is a 1-sink, then  $\alpha_a=\alpha+\alpha_1 \cdot \Dim P_1-\alpha_1 \cdot \Dim S_1.$
\end{itemize}
Moreover, in these cases we can see by direct computation that if $f=c^V$, then $f_a=c^{V'}$ is again a Schofield semi-invariant, where the representation $V'\in \Rep(Q_a,\alpha_a)$ can be written down explicitly. Since we will be working with generic Schur representations $V$, we will write only the corresponding dimension vectors (which are real Schur roots).

Writing $\sigma=\langle \alpha, \cdot \rangle = \langle \cdot, \alpha^* \rangle$, we can write down the dual formulas for the relation between $\alpha^*$ and $\alpha_a^*$ as well. They be deduced easily from the formulas above if we note that the dual semi-invariant $f^*$ on the opposite quiver $Q^*$ of $Q$ (i.e. reverse all arrows) has weight of $-\sigma = \langle \alpha^*, \cdot \rangle_*\,$, where $\langle \cdot, \cdot \rangle_*$ denotes the Euler product on $Q^*$.
\end{remark}

\vspace{0.1in}

\begin{definition} \label{def:sliceable}
For a semi-invariant $f$ of a quiver $Q$, we say $f$ is \itshape sliceable \normalfont if, after slicing repeatedly at $1$-sinks and $1$-arrows as described in Theorem \ref{thm:bquiv} (with possible simplifications, as in Lemma \ref{lem:simp}), we can reach the empty quiver (equivalently, a non-zero constant function).
\end{definition}
 %Due to Remark \ref{thm:import1}, if $(Q,f)$ is sliceable, then for the calculation of the $b$-function we can argue by induction and ignore the multiplicity-free requirement completely. 

 In the case $f$ is sliceable, we can compute the $b$-function and the locally semi-simple representation (see Proposition \ref{prop:locsemi}) of $f$ using the slice method. The following proposition gives a clearer picture of sliceable irreducible semi-invariants:

\begin{proposition}\label{prop:neg}
Let $f=c^V\in \op{SI}(Q,\beta)$ be an irreducible semi-invariant of weight $\langle \alpha, \cdot \rangle=-\langle \cdot, \alpha^* \rangle$ and assume $f$ depends on all arrows of $Q$. If $\alpha$ (resp. $\alpha^*$) is not a real Schur root, then $f$ is not sliceable.

Furthermore, take an arrow $\vec{a}$ that is a $1$-source or $1$-sink between $1$ and $2$ such that $\beta_1\leq \beta_2$, and assume $\alpha$ is a real Schur root. Let $\langle \alpha_a , \cdot \rangle$ be the weight of the induced semi-invariant $f_a$ on the slice $(Q_a,\beta_a)$, and let $\langle \alpha_a', \cdot \rangle$ be the weight on $(Q_a',\beta_a')$ after possible simplifications as in Lemma \ref{lem:simp}. Then the following are equivalent:
\begin{itemize}
\item[(a)] $\alpha_a$ is a real Schur root;
\item[(b)] $\alpha_a'$ is a real Schur root;
\item[(c)] $\vec{a}$ is a $1$-source with $\alpha_1=\alpha_2$ or $\alpha^*_1=0$, or $\vec{a}$ is a $1$-sink with $\alpha_1=0$ or $\alpha_1^*=\alpha_2^*$.
\end{itemize}
\end{proposition}

\begin{proof}
We will assume $a$ is a $1$-source (the case with $1$-sink is similar). Since $f$ depends on all arrows of $Q$ and is irreducible, we have by Theorem \ref{thm:bquiv} part a)  that $\beta$ and $\beta_a$ are sincere dimension vectors. Due to the isomorphism $\op{SI}(Q,\beta)\cong\complex[\Rep(Q_a,\beta_a)]^{U\rtimes\op{SL}(\beta_a)}$, we also have that $f_a=c^{V'}$ is irreducible. Since $\beta$ and $\beta_a$ are sincere, $V$ and $V'$ are Schur representations by \cite[Lemma 1]{harwey}. 

Note that $\langle \alpha , \alpha \rangle = -\langle \alpha, \alpha^*\rangle = \langle \alpha^*,\alpha^* \rangle$. By a direct computation, one obtains the formula
$$\langle \alpha_a, \alpha_a \rangle_a = \langle \alpha , \alpha \rangle -(\alpha_2-\alpha_1)\alpha_1^*,$$
where $\langle \cdot, \cdot \rangle_a$ is the Euler form on $Q_a$. This implies that this value decreases by slicing (at least before simplifications), and it remains the same iff $\alpha_2=\alpha_1$ or $\alpha_1^*=0$. However, we can simplify according to Lemma \ref{lem:simp} precisely under these conditions, and we get a reduced quiver $Q_a'$ with $\alpha_a'$. But an easy computation yields that the value $\langle \alpha_a' , \alpha_a' \rangle_a'=\langle \alpha_a, \alpha_a \rangle_a$ still remains the same. Since $V$ (resp. $V'$) are Schur representations, $\alpha$ (resp. $\alpha_a$) is a real Schur root if and only if $\langle \alpha, \alpha \rangle=1$ (resp. $\langle \alpha_a , \alpha_a \rangle = 1$). Now assume $f$ is sliceable. Since $V$ is a Schur representation, we have $\langle \alpha , \alpha \rangle \leq 1$. Since this value can only decrease by slicing and the last value (when the function is constant) is trivially $1$, we must have that all values are $1$, and the encountered dimension vectors are all real Schur roots.
\end{proof}

Finally, we summarize the rules of slicing in the most common situation described in part (c) of the above theorem, combining Lemma \ref{lem:simp}, Theorem \ref{thm:bquiv} and Remark \ref{rem:impo}.

\begin{corollary}\label{cor:rulz}
Take $Q$ and $f$ a semi-invariant of weight $\sigma=\langle \alpha,\cdot\rangle$ as in Theorem \ref{thm:bquiv}. Slicing at the arrow $\vec{a}$ in the following cases, we obtain the slice $(Q_a,\beta_a)$ and induced semi-invariant $f_a$ with weight $\sigma_a=\langle \alpha_a, \cdot \rangle$:
\begin{itemize}
\item[(a)] If $\vec{a}$ is a 1-source with $\alpha_1=\alpha_2$, then 
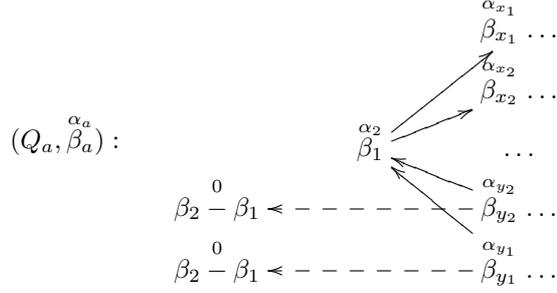
\[(Q_a,\stackrel{\alpha_a}{\beta_a}): \hspace{0.2in} \vcenter{\vbox{\xymatrix@R-2pc@C+0.5pc{
 & & \stackrel{\alpha_{x_1}}{\beta_{x_1}} \dots \\
 & & \stackrel{\alpha_{x_2}}{\beta_{x_2}}  \dots \\
   & \stackrel{\alpha_2}{\beta_1} \ar[uur] \ar[ur] &\dots \\
\stackrel{0}{\beta_2-\beta_1} & & \stackrel{\alpha_{y_2}}{\beta_{y_2}} \ar[ul] \ar@{-->}[ll] \dots\\
\stackrel{0}{\beta_2-\beta_1} & & \stackrel{\alpha_{y_1}}{\beta_{y_1}} \ar[uul] \ar@{-->}[ll] \dots
}}}\]
\item[(b)] If $\vec{a}$ is a 1-sink with $\alpha_1=0$, then
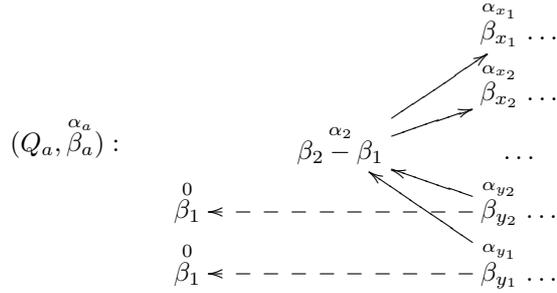
\[(Q_a,\stackrel{\alpha_a}{\beta_a}): \hspace{0.2in} \vcenter{\vbox{\xymatrix@R-2pc@C+0.5pc{
 & & \stackrel{\alpha_{x_1}}{\beta_{x_1}} \dots \\
 & & \stackrel{\alpha_{x_2}}{\beta_{x_2}}  \dots \\
   & \stackrel{\alpha_2}{\beta_{2}-\beta_1} \ar[uur] \ar[ur] &\dots \\
\stackrel{0}{\beta_1} & & \stackrel{\alpha_{y_2}}{\beta_{y_2}} \ar[ul] \ar@{-->}[ll] \dots\\
\stackrel{0}{\beta_1} & & \stackrel{\alpha_{y_1}}{\beta_{y_1}} \ar[uul] \ar@{-->}[ll] \dots
}}}\]
\end{itemize}
Moreover, writing $\sigma= -\langle \cdot, \alpha^* \rangle$, we have rules dual to the above by replacing $\alpha$ with $\alpha^*$, with all arrows reversed. Furthermore, in all these four cases $\alpha$ is a real Schur root if and only if $\alpha_a$ is a real Schur root, in which case 
\[b_f(s) = b_{f_a}(s) \cdot [s]_{\beta_2-\beta_1,\beta_2}^{\alpha_2}.\] 
\end{corollary}

\begin{remark} For a semi-invariant $f$ to be non-zero, some inequalities must be satisfied between the dimensions $\beta_x$, where $x\in Q_0$. The isomorphism $\op{SI}(Q,\beta)\cong\op{SI}(Q_a,\beta_a)^{U}$ from Proposition \ref{prop:algiso} gives inductively these inequalities, and they will be encoded in the negativity of the roots of the $b$-function. For simplicity, we will work with dimension vectors $\beta$ so that these inequalities are strict.
\end{remark}

\subsection{Some computations of $b$-functions for quivers}\label{subsec:exquiv}

We now show how to use Theorem \ref{thm:bquiv} and Corollary \ref{cor:rulz} in examples. We place the values of $\alpha$ or $\alpha^*$ on top of the values of the dimension vector $\beta$, where $\langle \alpha , \cdot \rangle = - \langle \cdot, \alpha^*\rangle$ is the weight of the semi-invariant. When  $\alpha^*$ is used, we label its values by $*$ at each vertex. We use a dashed line for the arrow at which we are slicing. We indicate (below the curly arrow) the slicing rule used from Corollary \ref{cor:rulz} (or Remark \ref{rem:impo} or Lemma \ref{lem:simp}) and retain (above the curly arrow) the decomposition of the $b$-function as given by Corollary \ref{cor:rulz} (or Theorem \ref{thm:bfun}).

\begin{example}\label{thm:beefor} We compute the $b$-function of the semi-invariant from Example \ref{eq:deefor}. Recall $\beta_1+\beta_2+\beta_3=2\beta_4$.
\[\vcenter{\vbox{\xymatrix{
 & \stackrel{1}{\beta_2} \ar[d] & \\
\stackrel{1}{\beta_1} \ar@{-->}[r] & \stackrel{1}{\beta_4} & \ar[l] \stackrel{1}{\beta_3}}
}}\stackrel[\ref{cor:rulz} (a)]{[s]_{\beta_4-\beta_1,\beta_4}}{\xymatrix{ \ar@{~>}[r] & }}
\vcenter{\vbox{\xymatrix{
 \stackrel{1}{\beta_2} \ar[d] \ar@{-->}[r] & \stackrel{0}{\beta_4-\beta_1} & \\
 \stackrel{1}{\beta_1} & \ar[l] \stackrel{1}{\beta_3} \ar[r] & \stackrel{0}{\beta_4-\beta_1}}
}}\stackrel[\ref{cor:rulz} (b)]{[s]_{\beta_1+\beta_2-\beta_4,\beta_2}}{\xymatrix@C+1pc{ \ar@{~>}[r] & }}\]
\[
\vcenter{\vbox{\xymatrix{
 \stackrel{1}{\beta_4-\beta_3} \ar@{-->}[d]  &  & \\
 \stackrel{1}{\beta_1} & \ar[l] \stackrel{1}{\beta_3} \ar[r] & \stackrel{0}{\beta_4-\beta_1}}
}}\stackrel[\ref{cor:rulz} (a)]{[s]_{\beta_1+\beta_3-\beta_4,\beta_1}}{\xymatrix@C+1pc{ \ar@{~>}[r] & }}
\vcenter{\vbox{\xymatrix@C-0.5pc{ & \stackrel{1}{\beta_4-\beta_3} & \\
  \stackrel{0}{\beta_4-\beta_2} & \ar[l]\ar[u] \stackrel{1}{\beta_3} \ar[r] & \stackrel{0}{\beta_4-\beta_1}}
}}\stackrel[\ref{lem:simp} (b)]{}{\xymatrix{ \ar@{~>}[r] & }}\]
\vspace{0.1in}
\[\vcenter{\vbox{\xymatrix{
  \stackrel{0}{\beta_4-\beta_2} & \ar@{-->}[l] \stackrel{1}{\beta_3} \ar[r] & \stackrel{0}{\beta_4-\beta_1}}}}
\stackrel[\ref{cor:rulz} (b)]{[s]_{\beta_2+\beta_3-\beta_4,\beta_3}}{\xymatrix@C+1pc{ \ar@{~>}[r] & }}
\vcenter{\vbox{\xymatrix{
\ar@{-->}[r] \stackrel{1}{\beta_4-\beta_1} & \stackrel{0}{\beta_4-\beta_1}}
}}\stackrel{[s]_{\beta_4-\beta_1}}{\xymatrix{ \ar@{~>}[r] & }} \emptyset
\]
Hence the $b$-function is 
$$b(s)=[s]_{\beta_4-\beta_1,\beta_4}\cdot[s]_{\beta_1+\beta_2-\beta_4,\beta_2}\cdot[s]_{\beta_1+\beta_3-\beta_4,\beta_1}\cdot[s]_{\beta_2+\beta_3-\beta_4,\beta_3}\cdot[s]_{\beta_4-\beta_1}=$$
$$=[s]_{\beta_4}\cdot[s]_{\beta_1+\beta_2-\beta_4,\beta_2}\cdot[s]_{\beta_2+\beta_3-\beta_4,\beta_3}\cdot[s]_{\beta_1+\beta_3-\beta_4,\beta_1}.$$

Using Proposition \ref{prop:locsemi} at each step, we get that the locally semi-simple representation is
$$A=V_1^{\beta_4-\beta_1}\oplus V_2^{\beta_4-\beta_2} \oplus V_3^{\beta_4-\beta_3},$$
where the indecomposables are $V_1= \begin{matrix} & 1 & \\ 0 & 1 & 1\end{matrix} $, $V_2 = \begin{matrix} & 0 & \\ 1 & 1 & 1\end{matrix}$ , $V_3 =\begin{matrix} & 1 & \\ 1 & 1 & 0\end{matrix}$. Note that this is also the generic representation in $\Rep(Q,\beta)$. This is due to the fact that $\Rep(Q,\beta)\backslash \orb_A$ is the hypersurface defined by the semi-invariant.
\end{example}

\vspace{0.1in}

Now we formulate a result for tree quivers, that is, for quivers whose underlying graphs have no cycles. This includes the $b$-functions of semi-invariants for type $\A$ quivers determined in \cite{sugi}.

\begin{theorem} \label{thm:tree} 
Let $Q$ be a tree quiver, and $f$ a non-zero semi-invariant on $Q$ of weight $\langle \alpha,\cdot \rangle=-\langle \cdot, \alpha^* \rangle$. If $\alpha_x\leq 1$ for any $x\in Q_0$ (resp. $\alpha_x^*\leq 1$ for any $x\in Q_0$), then $f$ is sliceable, and the roots of $b_f(s)$ are negative integers.
\end{theorem}

\begin{proof}
By duality, it is enough to consider the case $\alpha_x\leq 1$ for all $x\in Q_0$. It is immediate that $\alpha$ is a prehomogeneous dimension vector, hence the weight $\langle \alpha, \cdot \rangle$ is multiplicity-free. As usual, we work with the support of $f$, that is, we can drop arrows if $f$ doesn't depend on its corresponding variables. Since $Q$ is a tree, we can take an arrow $\vec{a}\in Q_1$ that is a 1-source or 1-sink. We use the notation as in Theorem \ref{thm:bquiv}.

First, assume $\vec{a}$ is 1-source. If $f$ depends on $\vec{a}$, we must have $\alpha_1=1$ by Lemma \ref{lem:simp}. Let $A$ be the generic matrix of variables corresponding to $\vec{a}$. If $\alpha_2=0$, then by Lemma \ref{lem:simp} part a) we can disconnect the quiver, $A$ has to be a square matrix, and we can separate variables $f=f'\cdot \det A$, where $f'$ is a semi-invariant on the smaller quiver without the arrow $\vec{a}$. Hence we can assume $\alpha_2=1$.  

Similarly, if $\vec{a}$ is a 1-sink, we can assume WLOG that $\alpha_1=0$ and $\alpha_2=1$.

In any case, we are in the situation of slicing at $\vec{a}$ as in Corollary \ref{cor:rulz}, and get a quiver $Q_a$ which is still a tree quiver, and the weight $\alpha_a$ of the induced semi-invariant $f_a$ on $Q_a$ still satisfies $(\alpha_a)_x\leq 1$, for any $x\in (Q_a)_0$. By Theorem \ref{thm:bquiv}, we get
$$b_f(s)=b_{f_a}(s)\cdot [s]_{\beta_2-\beta_1,\beta_2}.$$
Since the dimension of the representation space strictly decreases by slicing, this procedure is finite and stops when we arrive at a constant function.
\end{proof}

For some geometric implications of the result above about singularities of the zero sets of such semi-invariants, see \cite[Theorem 3.13]{en2}. We consider the next family of Dynkin quivers:
\begin{theorem}\label{thm:dee} 
All fundamental semi-invariants of quivers of type $\D_n$ are sliceable.
\end{theorem} 

\begin{proof}
Proceeding as in Theorem \ref{thm:tree} and using Corollary \ref{cor:rulz}, one we can reduce the proof to the case when $\alpha$ is the longest root. We illustrate the proof with the orientation of $\D_n$ chosen so that all arrows point to the joint vertex.

\[\vcenter{\vbox{\xymatrix@C-1.2pc@R-1pc{
  & & & & \stackrel{1}{\beta_{n-1}} \ar[d] & \\
\stackrel{1}{\beta_1}  \ar[r] & \stackrel{2}{\beta_2}  \ar[r]  & \stackrel{2}{\beta_3}  \ar[r]   & \dots \ar[r]&  \stackrel{2}{\beta_{n-2}} & \ar[l] \stackrel{1}{\beta_n}
}}}\stackrel{\alpha\to \alpha^*}{\xymatrix{ \ar@{~>}[r] & }}
\vcenter{\vbox{\xymatrix@C-1.1pc@R-1pc{
  & & & & \stackrel{1^*}{\beta_{n-1}} \ar[d] & \\
\stackrel{0^*}{\beta_1}  \ar@{-->}[r] & \stackrel{1^*}{\beta_2}  \ar[r]  & \stackrel{2^*}{\beta_3}  \ar[r]   & \dots \ar[r]&  \stackrel{2^*}{\beta_{n-2}} & \ar[l] \stackrel{1^*}{\beta_n}
}}}\stackrel[\ref{cor:rulz} (b)^*]{[s]_{\beta_2-\beta_1,\beta_2}}{\xymatrix{ \ar@{~>}[r] & }}\]
\[\vcenter{\vbox{\xymatrix@C-1.2pc@R-0.8pc{
 & \stackrel{0^*}{\beta_1}\ar@{-->}[d] & & & \stackrel{1^*}{\beta_{n-1}} \ar[d] & \\
\stackrel{1^*}{\beta_2-\beta_1}  \ar[r]  & \stackrel{2^*}{\beta_3}  \ar[r]   & \stackrel{2^*}{\beta_4}\ar[r] & \dots \ar[r]&  \stackrel{2^*}{\beta_{n-2}} & \ar[l] \stackrel{1^*}{\beta_n}
}}}\!\!\!\!\stackrel[\ref{cor:rulz} (b)^*]{[s]_{\beta_3-\beta_1,\beta_3}^2}{\xymatrix{ \ar@{~>}[r] & }}\!\!\!\!
\vcenter{\vbox{\xymatrix@C-1.2pc@R-0.8pc{
 & & \stackrel{0^*}{\beta_1}\ar@{-->}[d] & & \stackrel{1^*}{\beta_{n-1}} \ar[d] & \\
\stackrel{1^*}{\beta_2-\beta_1}  \ar[r]  & \stackrel{2^*}{\beta_3-\beta_1}  \ar[r]   & \stackrel{2^*}{\beta_4}\ar[r] & \dots \ar[r]&  \stackrel{2^*}{\beta_{n-2}} & \ar[l] \stackrel{1^*}{\beta_n}
}}}\]
\[\stackrel[\ref{cor:rulz} (b)^*]{[s]_{\beta_4-\beta_1,\beta_4}^2}{\xymatrix{ \ar@{~>}[r] & }} \dots \stackrel[\ref{cor:rulz} (b)^*]{[s]_{\beta_{n-3}-\beta_1,\beta_{n-3}}^2}{\xymatrix{ \ar@{~>}[r] & }}\!\!\!\!
\vcenter{\vbox{\xymatrix@C-1.2pc@R-1pc{
 & & & \stackrel{0^*}{\beta_1}\ar@{-->}[dr] & \stackrel{1^*}{\beta_{n-1}} \ar[d] & \\
\stackrel{1^*}{\beta_2-\beta_1}  \ar[r]  & \stackrel{2^*}{\beta_3-\beta_1}  \ar[r]   & \dots \ar[r] &\stackrel{2^*}{\beta_{n-3}-\beta_{1}}\ar[r] & \stackrel{2^*}{\beta_{n-2}} & \ar[l] \stackrel{1^*}{\beta_n}
}}}\!\!\!\stackrel[\ref{cor:rulz} (b)^*]{[s]_{\beta_{n-2}-\beta_1,\beta_{n-2}}^2}{\xymatrix{ \ar@{~>}[r] & }}\]
\[\vcenter{\vbox{\xymatrix@C-1.2pc@R-1pc{
 & & & \stackrel{1^*}{\beta_{n-1}} \ar[d] & \\
\stackrel{1^*}{\beta_2-\beta_1}  \ar[r]  & \stackrel{2^*}{\beta_3-\beta_1}  \ar[r]   & \dots \ar[r]&  \stackrel{2^*}{\beta_{n-2}-\beta_1} & \ar[l] \stackrel{1^*}{\beta_n}}}}
\!\!\!\!\!\!\!\stackrel{\alpha^*\to \alpha}{\xymatrix{ \ar@{~>}[r] & }}\!\!\!\!\!\!\!
\vcenter{\vbox{\xymatrix@C-1.2pc@R-1pc{
 & & & \stackrel{1}{\beta_{n-1}} \ar[d] & \\
\stackrel{1}{\beta_2-\beta_1}  \ar[r]  & \stackrel{1}{\beta_3-\beta_1}  \ar[r]   & \dots \ar[r] & \stackrel{1}{\beta_{n-2}-\beta_1} & \ar[l] \stackrel{1}{\beta_n}
}}}\]
At this stage we know that the latter quiver is sliceable, by Theorem \ref{thm:tree}. Continuing,
\[\stackrel[\ref{cor:rulz} (a)]{[s]_{\beta_3-\beta_2,\beta_3-\beta_1}}{\xymatrix@C+1pc{ \ar@{~>}[r] & }}
\vcenter{\vbox{\xymatrix@C-0.8pc{
& & & \stackrel{1}{\beta_{n-1}} \ar[d] & \\
 \stackrel{1}{\beta_2-\beta_1}  \ar@{-->}[r]  & \stackrel{1}{\beta_4-\beta_1} \ar[r] & \dots \ar[r] & \stackrel{1}{\beta_{n-2}-\beta_1} & \ar[l] \stackrel{1}{\beta_n}
}}}\stackrel[\ref{cor:rulz} (a)]{[s]_{\beta_4-\beta_2,\beta_4-\beta_1}}{\xymatrix@C+1pc{ \ar@{~>}[r] & }} \dots \]
\[
\dots \stackrel[\ref{cor:rulz} (a)]{[s]_{\beta_{n-3}-\beta_2,\beta_{n-3}-\beta_1}}{\xymatrix@C+1pc{ \ar@{~>}[r] & }} \vcenter{\vbox{\xymatrix{
 & \stackrel{1}{\beta_{n-1}} \ar[d] & \\
 \stackrel{1}{\beta_2-\beta_1} \ar[r] & \stackrel{1}{\beta_{n-2}-\beta_1} & \ar[l] \stackrel{1}{\beta_n}
}}}\stackrel[\text{Example \ref{thm:beefor}}]{}{\xymatrix@C+3pc{ \ar@{~>}[r] & }}\emptyset\] 
%{[s]_{\beta_{n-2}-\beta_2,\beta_{n-2}-\beta_1}[s]_{\beta_{n-2}-\beta_{n-1}-\beta_1,\beta_{n-2}-\beta_1}[s]_{\beta_n-\beta_{n-2}-\beta_1,\beta_{n-2}-\beta_1}[s]_{\beta_{n-2}-\beta_1}}%

Hence the $b$-function is:
$$b(s)=[s]_{\beta_2-\beta_1,\beta_2}\prod_{i=3}^{n-2}\left([s]_{\beta_i-\beta_1,\beta_i}^2 [s]_{\beta_i-\beta_2,\beta_i-\beta_1}\right) \cdot $$
$$\cdot [s]_{\beta_{n-2}-\beta_{n-1}-\beta_1,\beta_{n-2}-\beta_1}[s]_{\beta_{n-2}-\beta_{n}-\beta_1,\beta_{n-2}-\beta_1}[s]_{\beta_{n-2}-\beta_1}.$$
Accordingly, the homogeneous inequalities that are necessary and sufficient for the semi-invariant to be non-zero are:
$$\beta_1\leq \beta_2 \leq \beta_i, i=3,\dots,n-2,$$
$$\beta_{n-1},\beta_n \leq \beta_{n-2}-\beta_1.$$
If these inequalities are strict, then the semi-invariant is irreducible by Proposition \ref{prop:algiso}. Also, one can write down the corresponding locally semi-simple representation explicitly using Proposition \ref{prop:locsemi} in each step.

\end{proof}

We give an example of a quiver of extended Dynkin type:

\begin{example} We take $\overline{\D}_4$ with the dimension vector $\beta$, with $2\beta_1+\beta_2+\beta_3+\beta_4=3\beta_5$, semi-invariant (unique up to constant) $f=c^V$, where $\Dim V=\alpha=(2,1,1,1,2)$ is a real Schur root:
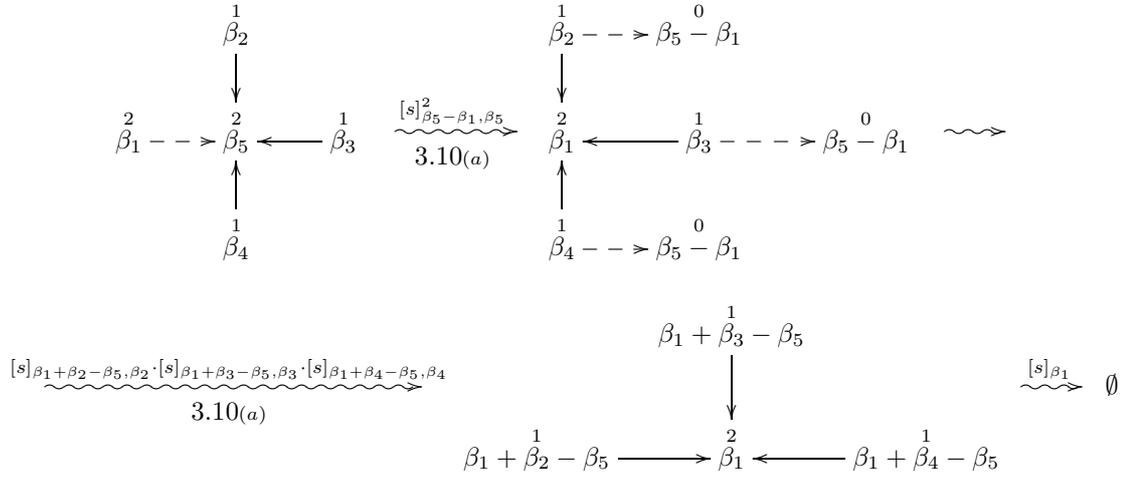
\[
\vcenter{\vbox{\xymatrix@R-0.5pc{
 & \stackrel{1}{\beta_2} \ar[d] & \\
\stackrel{2}{\beta_1} \ar@{-->}[r] & \stackrel{2}{\beta_5} & \ar[l] \stackrel{1}{\beta_3}\\
& \stackrel{1}{\beta_4} \ar[u] & 
}}}\stackrel[\ref{cor:rulz} (a)]{[s]^2_{\beta_5-\beta_1,\beta_5}}{\xymatrix@C+2pc{ \ar@{~>}[r] & }}
\vcenter{\vbox{\xymatrix@R-0.5pc{
\stackrel{1}{\beta_2} \ar[d] \ar@{-->}[r] & \stackrel{0}{\beta_5-\beta_1} & \\
\stackrel{2}{\beta_1} & \ar[l] \stackrel{1}{\beta_3} \ar@{-->}[r] & \stackrel{0}{\beta_5-\beta_1}\\
\stackrel{1}{\beta_4} \ar[u] \ar@{-->}[r] & \stackrel{0}{\beta_5-\beta_1} & 
}}}\stackrel{}{\xymatrix{ \ar@{~>}[r] & }}\]
\[\stackrel[\ref{cor:rulz} (a)]{[s]_{\beta_1+\beta_2-\beta_5,\beta_2}\cdot[s]_{\beta_1+\beta_3-\beta_5,\beta_3}\cdot[s]_{\beta_1+\beta_4-\beta_5,\beta_4}}{\xymatrix@C+10pc{ \ar@{~>}[r] & }}
\vcenter{\vbox{\xymatrix@C-1pc{
 & \stackrel{1}{\beta_1+\beta_3-\beta_5} \ar[d] & \\
\stackrel{1}{\beta_1+\beta_2-\beta_5} \ar[r] & \stackrel{2}{\beta_1} & \ar[l] \stackrel{1}{\beta_1+\beta_4-\beta_5}}
}}\!\!\!\stackrel{[s]_{\beta_1}}{\xymatrix{ \ar@{~>}[r] & }}
\emptyset
\]
In the last step we noticed the shortcut that the semi-invariant is just the square determinant of size $\beta_1$. So the $b$-function of $f$ is 
$$b_f(s)= [s]^2_{\beta_5-\beta_1,\beta_5} \cdot [s]_{\beta_1+\beta_2-\beta_5,\beta_2}\cdot[s]_{\beta_1+\beta_3-\beta_5,\beta_3}\cdot[s]_{\beta_1+\beta_4-\beta_5,\beta_4}\cdot[s]_{\beta_1}.$$
\end{example}

In contrast with the method by reflections from \cite{en}, we find a Dynkin quiver with a semi-invariant that is not sliceable.

\begin{example}\label{thm:noot}
Take the following quiver of type $\E_6$ with semi-invariant of weight $\langle \alpha, \cdot \rangle = -\langle \cdot, \alpha^*\rangle$, with $\alpha$ being the longest root:
\[\vcenter{\vbox{\xymatrix{
 & & 2\ar[d] & & \\
1 \ar[r]& 2\ar[r] & 3 & \ar[l] 2& \ar[l] 1
}}}\stackrel{\alpha \to \alpha^*}{\xymatrix{ \ar@{~>}[r] & }}
\vcenter{\vbox{\xymatrix{
 & & 1\ar[d] & & \\
1 \ar[r]& 2\ar[r] & 3 & \ar[l] 2& \ar[l] 1
}}}
\]
There are no $1$-sources (resp. $1$-sinks) $a$ with $\alpha_{ta}=\alpha_{ha}$ or with $\alpha^*_{ta}=0$ (resp. $\alpha^*_{ta}=\alpha^*_{ha}$ or $\alpha_{ha}=0$). By Proposition \ref{prop:neg} the semi-invariant is not sliceable. However, in order to compute the $b$-function one can apply the method by reflections from \cite{en}.
\end{example}

\begin{example}\label{ex:symm} Symmetric quivers.

Examples \ref{ex:symdet},\ref{ex:ortho},\ref{ex:symp} are particular cases of semi-invariants of \textit{symmetric} quivers, see \cite{semisymm,symmetric}. In \cite[Proposition 4.1]{sasu}, the $b$-function of a semi-invariant of the equioriented symmetric quiver of type $\A$ is computed based on the multiplicity one property. Many more $b$-functions of semi-invariants of symmetric quivers can be computed using the techniques developed in Section \ref{sec:slice}. A more systematic study of these will be pursued in a subsequent paper.
\end{example}

We show in the next example how to apply Theorem \ref{thm:bquiv} together with Theorem \ref{thm:multi} to compute $b$-functions of \textit{several variables}. The main difference in the process is that we can make only simultaneous simplifications for the semi-invariants as in Lemma \ref{lem:simp} or Corollary \ref{cor:rulz}.

\begin{example}\label{thm:several}($b$-function of several variables) Take the following $\D_5$ quiver with non-zero semi-invariants $f_i=c^{V_i}$, for $i=1,2$, $\alpha^1=\Dim V_1=(0,1,1,0,1)$, $\alpha^2=\Dim V_2=(1,1,0,0,0)$ and $\beta_1+\beta_4=\beta_3$, $\beta_2=\beta_5$. We put the values of $\alpha^1$ and $\alpha^2$ on top of $\beta$:
\[\vcenter{\vbox{\xymatrix{
  & & \stackrel{1,0}{\beta_3} \ar[d] & \\
\stackrel{0,1}{\beta_1}  & \stackrel{1,1}{\beta_2} \ar@{-->}[l]  \ar[r]  & \stackrel{1,0}{\beta_5} \ar[r] & \stackrel{0,0}{\beta_4}
}}}\stackrel[\ref{rem:impo} (b)]{[s]_{\beta_2-\beta_1,\beta_2}^{1,0}}{\xymatrix{ \ar@{~>}[r] & }}
\vcenter{\vbox{\xymatrix{
& \stackrel{1,0}{\beta_3} \ar[d] & \\
 \stackrel{1,1}{\beta_2-\beta_1} \ar@{-->}[r]  & \stackrel{1,1}{\beta_5} \ar[r] & \stackrel{0,1}{\beta_4}\\
 & \stackrel{0,1}{\beta_1} \ar[u] & 
}}}\stackrel[\ref{cor:rulz} (a)]{[s]_{\beta_1,\beta_5}^{1,1}}{\xymatrix{\ar@{~>}[r] & }}
\]
\[\vcenter{\vbox{\xymatrix@R-0.5pc@C-0.5pc{
   \stackrel{1,0}{\beta_3} \ar[d] \ar[r]&  \stackrel{0,0}{\beta_1}  \\
 \stackrel{1,1}{\beta_2-\beta_1} \ar[r] & \stackrel{0,1}{\beta_4}\\
 \stackrel{0,1}{\beta_1} \ar@{-->}[r] \ar[u] & \stackrel{0,0}{\beta_1} 
}}}\stackrel[\ref{cor:rulz} (b)]{[s]^{0,1}_{\beta_1}}{\xymatrix{ \ar@{~>}[r] & }}
\vcenter{\vbox{\xymatrix{
   \stackrel{1,0}{\beta_3} \ar[d] \ar@{-->}[r]&  \stackrel{0,0}{\beta_1}  \\
 \stackrel{1,0}{\beta_2-\beta_1} \ar[r] & \stackrel{0,0}{\beta_4}
}}}\stackrel[\ref{cor:rulz} (b)]{[s]^{1,0}_{\beta_4,\beta_3}}{\xymatrix{ \ar@{~>}[r] & }}
\vcenter{\vbox{\xymatrix{
   \stackrel{1,0}{\beta_4} \ar@{-->}[d] &  \\
 \stackrel{1,0}{\beta_2-\beta_1} \ar[r] & \stackrel{0,0}{\beta_4}
}}}\]
 \[ \stackrel[\ref{cor:rulz} (a)]{[s]^{1,0}_{\beta_2-\beta_3,\beta_2-\beta_1}}{\xymatrix{ \ar@{~>}[r] & }} 
\vcenter{\vbox{\xymatrix{
\stackrel{1,0}{\beta_4} \ar@{-->}[r] & \stackrel{1,0}{\beta_4} 
}}}\stackrel{[s]^{1,0}_{\beta_4}}{\xymatrix{ \ar@{~>}[r] & }}\emptyset\]

Hence we have
$$b_{\underline{m}}(s_1,s_2)=[s]_{\beta_1,\beta_2}^{1,1}[s]^{0,1}_{\beta_1}[s]^{1,0}_{\beta_3}[s]^{1,0}_{\beta_2-\beta_3,\beta_2}.$$
\end{example}

It is not difficult to see that for the quivers of type $\A_n$ the slice method is sufficient to compute all the $b$-functions of several variables (as are the methods in \cite{en},\cite{sugi}). However, the slice method is not always sufficient to obtain directly the $b$-functions of several variables for type $\D_n$ quivers (although the method in \cite{en} is). Nevertheless, given the individual $b$-function (of one variable) of each semi-invariant (see Theorem \ref{thm:dee}), one can in principle apply the Structure Theorem of $b$-functions as in \cite{sugi} for this purpose -- for an example, see \cite[Example 4.3.13]{phd}. To proceed as in \cite{sugi}, one needs an explicit description for the locally semi-simple representation (and use \cite[Lemma 4.2.4]{phd}) and the generic representation. One can describe the locally semi-simple representation of each semi-invariant by Proposition \ref{prop:locsemi}, and the generic representation using the procedure we present in Appendix \ref{app:decomp}.

\appendix
\section{Generic decomposition for Dynkin quivers of type $\D$} \label{app:decomp} 

Based on slices, we give an easy procedure for determining the generic decomposition for type $\D_n$ quivers. 

Let $Q$ be a quiver, and $\alpha$ a prehomogeneous dimension vector. Following \cite{kac2}, we call a decomposition
$$\alpha=\alpha_1 \oplus \alpha_2 \oplus \dots \oplus \alpha_t$$
the \itshape generic decomposition \normalfont (also called canonical decomposition), if the generic representation of dimension vector $\alpha$ decomposes into indecomposable representations of dimension vectors $\alpha_1,\alpha_2,\dots \alpha_t$. As already discussed in Section \ref{sec:quiv}, in this case $\alpha_i$ are real Schur roots, with $\Ext_Q(\alpha_i,\alpha_j)=0$ (that is, the corresponding generic representations have no self-extensions). Moreover, rewriting
$$\alpha=\alpha_1^{\oplus r_1}\oplus \alpha_2^{\oplus r_2} \oplus \dots \oplus \alpha_t^{\oplus r_t}.$$
with $\alpha_i$ distinct, we may assume, after a suitable rearrangement, that $\Hom_Q(\alpha_i,\alpha_j)=0$, for $i<j$ (again, this means that there are no morphisms between the corresponding generic representations). For more details , see \cite{combi,kac2}.

Though there exist algorithms to determine the generic decomposition for a dimension vector (e.g. see \cite{canon}), it is of interest to give clear-cut procedures that are easy to work out by hand. There is such a rule for quivers of type $\A_n$, and this is described in \cite[Proposition 3.1]{abeasis}. We illustrate this construction by the following example:
\[\xymatrix{
3 \ar[r] & 5 & 6 \ar[l] \ar[r] & 3 \ar[r] & 5
}\]
The generic decomposition is given by the following diagram (the connected horizontal components are the indecomposables):
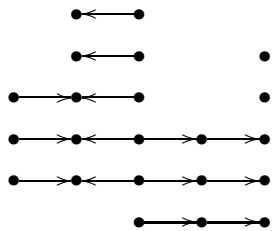
\[\xymatrix@C-0.4pc@R-1pc@M-0.28pc{
 & \bullet & \bullet\ar[l] & & \\
 & \bullet & \bullet\ar[l] & & \bullet \\
\bullet\ar[r] & \bullet  & \bullet\ar[l] & & \bullet\\
\bullet \ar[r] & \bullet & \bullet \ar[l]\ar[r] & \bullet \ar[r] & \bullet\\
\bullet \ar[r] & \bullet & \bullet \ar[l]\ar[r] & \bullet \ar[r] & \bullet\\
 & & \bullet\ar[r] & \bullet \ar[r] & \bullet
}\]

Based on the $\A_n$ case, we extend the rule for quivers of type $\D_n$. Take a quiver with underlying graph $\D_n$ and the following labeling:
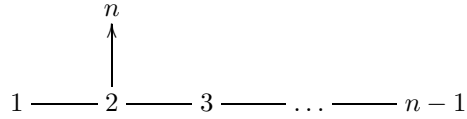
\[\xymatrix{
& n & &  & \\
1 & 2 \ar@{-}[l]\ar[u]\ar@{-}[r] & 3 \ar@{-}[r] & \dots \ar@{-}[r] & n-1
}\]
Since the generic decomposition of a quiver and its opposite quiver coincide, we will fix without loss of generality the orientation of the arrow $2\to n$. We illustrate the procedure by examples first. Take the following $\D_n$ quiver with $n=6$ and $\alpha=(3,5,6,3,5,4)$:
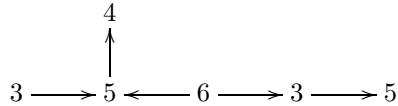
\[\xymatrix@R-0.5pc{& 4 &  &  & \\
3 \ar[r] & 5 \ar[u] & 6 \ar[l] \ar[r] & 3 \ar[r] & 5
}\]

First, take the generic decomposition of the $\A_{n-1}$ quiver by dropping the $n$-th vertex. This was done in the example above. Then, the indecomposables of $\A_{n-1}$ that have $0$ dimension at vertex $2$ will also appear in the generic decomposition for $\D_n$. Hence we drop them, and we are left with the following diagram:
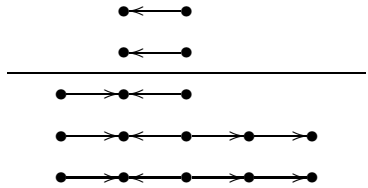
\[\xymatrix@C-0.4pc@R-1pc@M-0.28pc{
 & \bullet & \bullet\ar[l] & & \\
 & \bullet & \bullet\ar[l] & & \\
\bullet\ar[r]\ar@{-}[]+<-2em,0.8em>;[rrrr]+<2em,0.8em> & \bullet  & \bullet\ar[l] & & \\
\bullet \ar[r] & \bullet & \bullet \ar[l]\ar[r] & \bullet \ar[r] & \bullet\\
\bullet \ar[r] & \bullet & \bullet \ar[l]\ar[r] & \bullet \ar[r] & \bullet
}\]
We separated by a horizontal line the two classes of indecomposables with dimension at vertex $1$ equal to $0$ or equal to $1$. We call the indecomposables under this line of the first class and over the line of the second class.  Now we place $\alpha_n$ symbols $\circ$ on the left of the diagram starting from the horizontal line and moving downwards ($\circ$ represents the simple representation $S_n$). When we stop, we put another horizontal line to the bottom. Then we move the indecomposables of the second class starting from the top of the diagram and add their dimension vectors starting from the bottom horizontal line and stop if either:
\begin{itemize}
\item[(a)] We reach the top horizontal line, or
\item[(b)] We run out of indecomposables of the second class, or
\item[(c)] There exists a non-zero morphism from the indecomposable of the second class that we want to move to corresponding indecomposable of the first class.
\end{itemize}

In this example we stop due to part (b) and the diagram we get is:
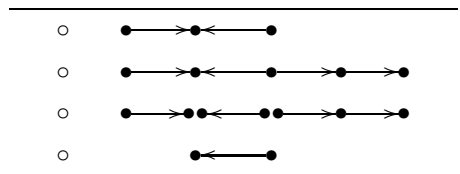
\[\xymatrix@C-0.4pc@R-1pc@M-0.28pc{
\ar@{-}[]+<-2em,0.8em>;[rrrrr]+<2em,0.8em> \circ &\bullet\ar[r] & \bullet  & \bullet\ar[l] & & \\
\circ &\bullet \ar[r] & \bullet & \bullet \ar[l]\ar[r] & \bullet \ar[r] & \bullet\\
\circ &\bullet \ar[r] & \bullet\bullet & \bullet\bullet\ar[l]\ar[r] & \bullet \ar[r] & \bullet\\
\ar@{-}[]+<-2em,-0.8em>;[rrrrr]+<2em,-0.8em> \circ & & \bullet & \bullet\ar[l]  & &
}\]

\vspace{0.1in}

Now we are ready to read off the generic decomposition. The indecomposables outside the horizontal lines will stay the same (there are none in this example). Finally, for each row between the two horizontal lines the dimension vector will have dimension $1$ at vertex $n$. Hence we get in this case
\begin{multline*}
(3,5,6,3,5,4)=(1,1,1,0,0,1)\oplus(1,1,1,1,1,1)\oplus(1,2,2,1,1,1)\oplus(0,1,0,0,0,1)\oplus\\
\oplus (0,0,0,0,1,0)^{\oplus 2} \oplus (0,0,1,1,1,0).
\end{multline*}

\vspace{0.1in}
We give another example:
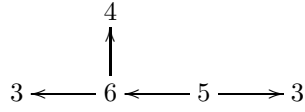
\[\xymatrix@R-0.5pc{& 4 &  &  & \\
3  & 6 \ar[l] \ar[u] & 5 \ar[l] \ar[r] & 3
}\]

The generic decomposition for the $\A_4$ part is 
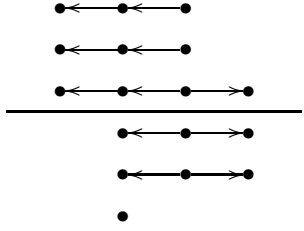
\[\xymatrix@C-0.4pc@R-1pc@M-0.28pc{
\bullet & \ar[l] \bullet & \bullet\ar[l] & \\
\bullet & \ar[l] \bullet & \bullet\ar[l] & \\
\bullet & \ar[l] \bullet & \bullet\ar[l] \ar[r] & \bullet\\
\ar@{-}[]+<-2em,0.8em>;[rrr]+<2em,0.8em> & \bullet & \bullet\ar[l] \ar[r] & \bullet\\
& \bullet & \bullet\ar[l] \ar[r] & \bullet\\
 & \bullet & & \\
}\]
Note that all indecomposables have dimension $1$ at the vertex $2$. The diagram joining the two classes of indecomposables is:
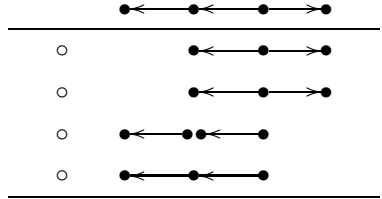
\[\hspace{-0.25in}\xymatrix@C-0.4pc@R-1pc@M-0.28pc{
& \bullet & \ar[l] \bullet & \bullet\ar[l] \ar[r] & \bullet\\
\ar@{-}[]+<-2em,0.8em>;[rrrr]+<2em,0.8em>  \circ & &\bullet & \bullet\ar[l] \ar[r] & \bullet\\
 \circ & & \bullet & \bullet\ar[l] \ar[r] & \bullet\\
 \circ & \bullet & \ar[l] \bullet \bullet & \bullet\ar[l] & \\
\ar@{-}[]+<-2em,-0.8em>;[rrrr]+<2em,-0.8em>  \circ & \bullet & \ar[l] \bullet & \bullet\ar[l] &
}\]

\vspace{0.1in}

Here we stopped due to condition (c) since there is a non-zero map from the indecomposable $1\leftarrow 1 \leftarrow 1 \rightarrow 1$ to the corresponding indecomposable $0\leftarrow 1 \leftarrow 1 \rightarrow 1$. Hence the generic decomposition is
$$(3,6,5,3,4)=(1,1,1,1,0)\oplus(0,1,1,1,1)^{\oplus 2}\oplus (1,2,1,0,1)\oplus (1,1,1,0,1).$$

\vspace{0.1in}

\begin{theorem} The algorithm described above gives the generic decomposition for $\D_n$ quivers.
\end{theorem}

\begin{proof}
We give a proof using slices. First, write the generic decomposition for a generic representation $R$ of the $\A_{n-1}$ quiver in the form
$$R=\bigoplus_{i=1}^m V_i^{p_i} \oplus \bigoplus_{i=1}^n W_i^{q_i} \oplus \bigoplus_i Z_i.$$
Here $V_i$ and $W_i$ are representations of the first and second class, respectively (separated by the horizontal line as in the examples) and $Z_i$ are the representations with dimension $0$ at vertex $2$. We assume that the order is chosen such that:
\begin{itemize}
\item[(a)] There is a map from $V_i$ to $V_j$ iff $j \leq i$;
\item[(b)] There is a map from $W_i$ to $V_j$ iff $j \leq i$;
\item[(c)] There are no maps from $V_i$ to $W_j$ for all $i,j$.
\end{itemize}
We note that this can be achieved immediately from the generic decomposition algorithm for $\A_{n-1}$ (after dropping the representations $Z_i$): $V_i$ are the representations below the horizontal line, ordered from top to bottom, and $W_i$ are the representations above the horizontal line, ordered from top to bottom. With this in mind, we take the slice as in Section \ref{subsec:slice}. Take a representation of the form $V=Z+R$ in $\Rep(\D_n,\alpha)$, with $Z\in \Hom(\complex^{\alpha_2},\complex^{\alpha_n})$. Then $V$ has a dense $\GL(\alpha)$-orbit if and only if $Z$ has a dense orbit in $\Hom(\complex^{\alpha_2},\complex^{\alpha_n})$ under the action of the stabilizer $G_R=\GL(\alpha_n)\times \GL(\underline{p})\times\GL(\underline{q}) \times U \times U'$, where $U=\prod_{j < i} \Hom(\complex^{p_i},\complex^{p_j}) \prod_{j < i} \Hom(\complex^{q_i},\complex^{q_j})$ and $U'=\prod_{i,j} \Hom(W_i,V_j)^{p_jq_i}$. It can be easily seen that forgetting about the action of $U'$, the following element already has a dense orbit in $\Hom(\complex^{\alpha_2},\complex^{\alpha_n})$:
$$Z=\kbordermatrix{
~ & V_1 & V_1 & \dots & V_m & \vrule & W_1 & W_1 & \dots & W_n\cr
& 1 & 0 & \dots & 0 &\vrule &  &  & &\cr
& 0 & 1 & \dots & 0 &\vrule  &  & &\cr
&  & & \ddots & & \vrule &  & & \iddots & \cr
& &  &  & & \vrule & 0 & 1 & \dots & 0\cr
& &  &  & & \vrule & 1 & 0 & \dots & 0
}.$$
Here there are $p_i$ (resp. $q_i$) columns corresponding to $V_i$ (resp. $W_i$), and we put the ones diagonally in the first (resp. second) block starting from the top left (resp. bottom left) until we reach the bottom or right (resp. top or right) edge of the block. The arrangement of ones corresponds to stopping under condition (a) or (b). Now using the action of $U'$, if two ones are in the same row corresponding to the columns of $V_i$ and $W_j$, and $\Hom_Q(W_j,V_i)\neq 0$, then we can cancel the $1$ in the column of $W_j$. This corresponds to stopping under condition (c).
\end{proof}

\begin{remark}
The article \cite{abeasisd} describes the generic decomposition for an equioriented quiver of type $\D$. The explicit description of generic representations for type $\A$ and $\D$ quivers is also pursued in the recent paper \cite{riedtrec}.
\end{remark}

\vspace{0.05in}

\section*{Acknowledgement}
The author is indebted to Prof. Jerzy Weyman for many valuable discussions.

\bibliographystyle{amsplain}
\bibliography{biblo}

\end{document}